\documentclass[a4paper,11pt,leqno]{amsart}
\usepackage[big]{norms_style}

\title{Partial Resolutions and Noncrossing Combinatorics}

\author[Minh-T\^{a}m Trinh]{Minh-T\^{a}m Quang Trinh}
\address{Department of Mathematics, Howard University, Washington, DC 20059}
\email{minhtam.trinh@howard.edu}

\author{Nathan Williams}
\address{Mathematical Sciences Department, University of Texas at Dallas, 800 West Campbell Road, Richardson, TX 75080}
\email{nathan.f.williams@gmail.com}

\begin{document}

\begin{abstract}
For any finite reductive group, we compute the central elements in its Hecke algebra that arise from partial Springer resolutions via the Harish-Chandra transform.
Of the two kinds of partial resolution, the larger is the more interesting case.
We deduce formulas for associated Hecke traces, generalizing work of Wan--Wang beyond type $A$, and Deodhar-like decompositions of braid varieties that map to partial Springer resolutions.
From the latter, we construct noncrossing sets that interpolate between rational Catalan and parking objects, generalizing our work with Galashin--Lam.
In parallel, we establish new formulas for arbitrary $a$-degrees of the HOMFLYPT invariants of positive braid closures, from which we construct noncrossing sets for rational Kirkman numbers.
\end{abstract}

\maketitle

\thispagestyle{empty}



\section{Introduction}

\subsection{}\label{subsec:hoefsmit-scott}

Fix a finite Coxeter system $(W, S)$ and a subset $J \subseteq S$ generating a subgroup $W_J \subseteq W$.
Let $H_W$ and $H_{W_J}$ be the Hecke algebras over $\bb{Z}[\X^{\pm 1}]$ corresponding to $W$ and $W_J$.
We take the convention where the Hecke operators $T_s \in H_W$, for $s \in S$, obey the relations $T_s^2 = (\X - 1)T_s + \X$.
We identify $H_{W_J}$ with the subalgebra of $H_W$ generated by the elements $T_s$ with $s \in J$.
Our starting point is the existence of two separate ways to construct elements of the center $Z(H_W)$ from $J$.

First, according to L.\ K.\ Jones \cite{jones-l}, there is an injective, linear map
\begin{align}
	N_J^S : Z(H_{W_J}) \to Z(H_W)
\end{align}
due to Hoefsmit--Scott, called the \dfemph{relative norm}.
To define $N_J^S$, recall that each right coset of $W_J$ in $W$ contains a unique representative of minimal Bruhat length.
Let $W^J$ be the set of such representatives.
Then
\begin{align}
	N_J^S(\alpha) = \sum_{v \in W^J} \X^{-\ell(v)} T_{v^{-1}} \alpha T_{v\vphantom{^{-1}}}
	\quad\text{for all $\alpha \in Z(H_{W_J})$},
\end{align}
where $T_v$ and $\ell(v)$ denote the Hecke operator and Bruhat length of $v$.

Second, when $W$ is crystallographic, we can interpret it as the Weyl group of a split finite reductive group $G$ with Borel $B$.
We can then interpret $H_W$ and $H_{W_J}$ geometrically, as convolution algebras of functions on $(G/B)^2$.
Here, another way to produce elements of $Z(H_W)$ is a certain map from functions on $G$ to functions on $(G/B)^2$, which we call the Harish-Chandra transform, following \cite{ginzburg}.

The main observation of this paper is that the two partial Springer resolutions attached to $J$, as defined in \eqref{eq:springer}, produce functions on $G$ whose Harish-Chandra transforms are relative norms.
Related but different observations appeared in \cite{grojnowski, lusztig_15, bt}, as we discuss in \Cref{rem:prior}.
We give applications to:
\begin{enumerate}
\item 	Traces on $H_W$, generalizing work of Lascoux \cite{lascoux} and Wan--Wang \cite{ww}.

\item 	Deodhar-like decompositions of \dfemph{partial braid Steinberg varieties}, generalizing work of Shende--Treumann--Zaslow \cite{stz} and resembling Mellit's decompositions of character varieties \cite{mellit}.

\item 	The rational noncrossing combinatorics of $(W, S)$, generalizing our prior work with Galashin--Lam \cite{gltw}.

\end{enumerate} 
Our point of view also leads us to new formulas for the bivariate Hecke traces used to construct the HOMFLYPT  link invariant, which simplify formulas from \cite{bt} and give applications to rational Kirkman numbers.

\subsection{The Harish-Chandra Transform}\label{subsec:setup}

Let $\bb{F}$ be a finite field of order $q$.
Let $\bb{G}$ be a connected reductive algebraic group over $\bar{\bb{F}}$, equipped with an $\bb{F}$-structure corresponding to a Frobenius map $F : \bb{G} \to \bb{G}$.
We assume that the characteristic of $\bb{F}$ is a good prime for $\bb{G}$ \cite[28]{carter}.

Fix an $F$-stable maximal torus in an $F$-stable Borel: $\bb{T} \subseteq \bb{B} \subseteq \bb{G}$.
Let $\bb{W} = N_\bb{G}(\bb{T})/\bb{T}$.
We now take $W$ to be the finite Coxeter group $\bb{W}^F$.
Similarly, we write $G$, $B$, \emph{etc.}\ for the groups formed by the $F$-fixed points of $\bb{G}$, $\bb{B}$, \emph{etc.}

The $G$-invariant, $\bb{Z}[\tfrac{1}{q}]$-valued functions on $(G/B)^2$ form a convolution algebra $H_B^G$.
If $G$ is \dfemph{split}, meaning $W = \bb{W}$, then $H_B^G$ is the specialization at $\X \to q$ of the algebra $H_W$ presented earlier.
Explicitly, $T_w$ specializes to the indicator function on the set of pairs $(hB, gB)$ such that $Bh^{-1}gB = BwB$.
In \Cref{sec:hecke}, we review the presentation of $H_B^G$ for general $G$.

The \dfemph{Harish-Chandra transform} is the map $\HC_!$ from class functions on $G$ to invariant functions on $(G/B)^2$ given by pullback, then pushforward, through the diagram
\begin{align}
	G \xleftarrow{\pr_2} G/B \times G \xrightarrow{\act} G/B \times G/B,
\end{align}
where above,
\begin{align}
\act(hB, z) = (hB, zhB)
	\quad\text{and}\quad
\pr_2(hB, z) = z.
\end{align}
In the notation of \S\ref{subsec:functions}, $\HC_! \vcentcolon= \act_!\, \pr_2^\ast$.\footnote{Strictly speaking, our Harish-Chandra transform is Verdier dual to the one in \cite{ginzburg, bt}.}
A purely formal argument shows that $\HC_!$ takes values in the center of $H_B^G$ \cite[\S{9}]{ginzburg}.
It turns out that every central element arises this way, up to scaling \cite{lusztig_15}.

\subsection{Two Partial Resolutions}

\emph{In the rest of this introduction, we assume that $G$ is split, for simplicity.}
We take $S$ to be the system of simple reflections arising from $\bb{B}$.
Let $\bb{P}_J = \bb{B}W_J\bb{B}$, a parabolic subgroup of $\bb{G}$.
Let $\bb{U}_J$ be its unipotent radical and $\bb{V}_J$ the variety of all unipotent elements in $\bb{P}_J$.
At the level of points, the two \dfemph{partial Springer resolutions} of type $J$ are defined by
\begin{align}\begin{split}\label{eq:springer}
	\bb{Spr}_J^+ 
	&= \{(u, y\bb{P}_J) \in \bb{G} \times \bb{G}/\bb{P}_J \mid u \in y\bb{V}_Jy^{-1}\},\\
	\bb{Spr}_J^- 
	&= \{(u, y\bb{P}_J) \in \bb{G} \times \bb{G}/\bb{P}_J \mid u \in y\bb{U}_Jy^{-1}\}.
\end{split}\end{align}
If $J = \emptyset$, then $\bb{P}_J = \bb{B}$, giving $\bb{U}_J = \bb{V}_J$ and $\bb{Spr}_\emptyset^- = \bb{Spr}_\emptyset^+$. 
This is the original Springer resolution.
In general, the $+$ case is a partial resolution of singularities of the unipotent variety $\bb{V} \subseteq \bb{G}$, while the $-$ case is a resolution of the closure of the Richardson orbit for $J$.
It will be convenient to set $\bb{E}_J^\pm \vcentcolon= \bb{G}/\bb{B} \times \bb{Spr}_J^\pm$.

For any $G$-equivariant map $\pi : E \to X$, we write $\pi_! \sf{1}_E$ for the function on $X$ defined by $\pi_! \sf{1}_E(x) = |\pi^{-1}(x)|$.
Applying $\HC_!$ to the functions $\pr_{1, !} \sf{1}_{\mathit{Spr}_J^\pm}$, where $\mathit{Spr}_J^\pm = (\bb{Spr}_J^\pm)^F$, yields the functions $\mult_! \sf{1}_{E_J^\pm}$, where $E_J^\pm = (\bb{E}_J^\pm)^F$ and
\begin{align}
\mult(hB, u, yP_J) \vcentcolon= \act(hB, u) = (hB, uhB).
\end{align}
These are the elements that we will calculate in \Cref{sec:springer}.

Let $w_\circ$ and $w_{J\circ}$ respectively denote the longest elements of $W$ and $W_J$.
For convenience, we set $\ell_S = \ell(w_\circ)$ and $\ell_J = \ell(w_{J\circ})$.
Recall that $w_\circ, w_{J\circ}$ are involutions, and that $T_{w_{J\circ}}^2$ is central in $H_{W_J}$ \cite{bmr}.
The split case of our main result is:

\begin{thm}\label[thm]{thm:main}
For any $J \subseteq S$, we have
\begin{align} 
\mult_!\sf{1}_{E_J^-}
	= q^{\ell_S - \ell_J} N_J^S(1)|_{\X \to q}
	\quad\text{and}\quad
\mult_!\sf{1}_{E_J^+}
	= q^{\ell_S - \ell_J} N_J^S(T_{w_{J\circ}}^2)|_{\X \to q}.
\end{align} 
\end{thm}

Let $W^{J, -} = W^J$ be the set of minimal-length representatives for the right cosets of $W_J$ in $W$, and by analogy, let $W^{J, +}$ be the set of \emph{maximal-length} representatives, so that multiplication by $w_{J\circ}$ interchanges $W^{J, -}$ with $W^{J, +}$.
Then the above identities of central elements can be rewritten as:
\begin{align}
\left.\begin{array}{r@{\:}l}
\mult_! \sf{1}_{E_J^-} 
	&= q^{\ell_S - \ell_J} \Sigma_{J, -}|_{\X \to q},\\
\mult_! \sf{1}_{E_J^+} 
	&= q^{\ell_S} \Sigma_{J, +}|_{\X \to q},
	\end{array}\right\} 
	\quad\text{where $\Sigma_{J, \pm} = \displaystyle\sum_{v \in W^{J, \pm}} \X^{-\ell(v)} T_{v^{-1}} T_{v\vphantom{^{-1}}}$}.
\end{align}
We emphasize that the $+$ case is deeper than the $-$ case.
The $-$ case only uses standard results about Bruhat decomposition.
Under the assumption that $G$ is split, we can refine it to an algebro-geometric statement about $\bb{E}_J^-$:
See \Cref{prop:minus-case}.
By contrast, the $+$ case relies on a difficult theorem of Kawanaka \cite{kawanaka}.
We expect that it can only be refined to a statement at the level of cohomology.
This issue is related to the refinement of Kawanaka's work discussed in \cite{trinh_duality}.

\begin{rem}\label[rem]{rem:prior}
The $-$ case of \Cref{thm:main} is related to several results in the literature, though the statements in both cases are new to the best of our knowledge.

In Grojnowski's thesis \cite{grojnowski}, the proof of Proposition 2.1 can be used to recover the $J = \emptyset$ case of \Cref{thm:main}.
See also (3.3.1) in \cite{grojnowski}.
Note that Grojnowski works with the image of $H_W$ in the \dfemph{monodromic Hecke algebra} of functions on $Y_\emptyset \vcentcolon= \bb{Y}_\emptyset^F$, where in general $\bb{Y}_J \vcentcolon= (\bb{G}/\bb{U}_J)^2/\bb{L}_J$ with $\bb{L}_J$ being the Levi factor of $\bb{P}_J$ acting from the right.

In \cite{lusztig_15}, Lusztig studies an endofunctor of the derived category of constructible mixed complexes on $(\bb{G}/\bb{B})^2$.
It categorifies $\HC_!\, \CH_!$ on a suitable subcategory, where $\CH_! \vcentcolon= \pr_{2,!}\, \act^\ast$ in the notation of \S\ref{subsec:functions}.
Lusztig's Proposition 2.6 shows that his functor sends any $K$ into the triangulated category generated by objects of the form
\begin{align} 
\bigoplus_{\ell(v) = k} K_{v^{-1}} \ast K \ast K_{v\vphantom{^{-1}}} \otimes \mathfrak{L}[\![k]\!],
\end{align}
where $K_w$ and $\mathfrak{L}[\![k]\!]$ categorify $\sf{1}_w$ and $|T| q^{-k}$, respectively.

In \cite[\S{6}]{bt}, Bezrukavnikov--Tolmachov study a functor that categorifies $\HC_!\, \chi_{P_J}^G$, where $\chi_{P_J}^G$ sends $G$-invariant functions on $Y_J \vcentcolon= \bb{Y}_J^F$ to class functions on $G$.
When $J = \emptyset$, their setup specializes to a monodromic version of Lusztig's setup.
In this sense, their Lemma 6.2.6 generalizes \cite[Prop.\ 2.6]{lusztig_15} to arbitrary $J$.
Its proof is similar to our proof of the $-$ case of \Cref{thm:main}, though not of the $+$ case.
\end{rem} 

\subsection{Traces}\label{subsec:trace-intro}

A \dfemph{trace} on an algebra is a linear map that vanishes on commutators.
We write $R(H_W)$ for the vector space of $\bb{Q}(\X)$-valued traces on $H_W$.
Our first application of \Cref{thm:main} is to identify certain elements of $R(H_W)$ arising from $\Sigma_{J, \pm}$.

Let $e \in W$ be the identity.
Let $\tau : H_W \to \bb{Z}[\X^{\pm 1}]$ be the trace given by $\tau(T_e) = 1$ and $\tau(T_w) = 0$ for all $w \neq e$.
Then any central element $\zeta \in Z(H_W)$ gives rise to a trace $\tau[\zeta] : H_W \to \bb{Z}[\X^{\pm 1}] \subseteq \bb{Q}(\X)$: namely, 
\begin{align} 
	\tau[\zeta](\beta) = \tau(\beta\zeta).
\end{align} 
These traces are best understood when $W$ is a symmetric group.

Let $S_n$, the symmetric group on $n$ letters, and let $\Lambda_n$ be the vector space of symmetric functions over $\bb{Q}(\X)$ of degree $n$ in variables $X \vcentcolon= (x_1, \ldots, x_n)$.
Then $R(H_{S_n})$ is isomorphic to $\Lambda_n$, as both of these vector spaces have bases indexed by the integer partitions of $n$.
Let $\FC_\X : R(H_{S_n}) \xrightarrow{\sim} \Lambda_n$ be the \dfemph{$\X$-deformed Frobenius characteristic} isomorphism that sends the irreducible character $\chi_\X^\lambda$ indexed by a partition $\lambda \vdash n$ to the Schur function $s_\lambda[X]$.

For $W = S_n$, we take $S = \{s_1, \ldots, s_{n - 1}\}$, where $s_i = (i, i + 1)$.
This choice sets up a bijection between subsets $J \subseteq S$ and integer \emph{compositions} $\nu$ of $n$.
Let $e_\nu[X]$ and $h_\nu[X]$ respectively denote the elementary and complete homogeneous symmetric functions in $\Lambda_n$ indexed by $\nu$.
Wan--Wang \cite{ww}, recasting work of Lascoux \cite{lascoux}, show that if $J$ corresponds to $\nu$, then
\begin{align}\begin{split}\label{eq:lascoux}
\FC_\X(\tau[\Sigma_{J, -}]) 
	&= (\X - 1)^n e_\nu\left[\frac{X}{\X - 1}\right],\\
\FC_\X(\tau[\Sigma_{J, +}]) 
	&= (\X - 1)^n h_\nu\left[\frac{X}{\X - 1}\right].
\end{split}\end{align}
Using these identities, they show that the relative norm maps $N_J^S$ from \S\ref{subsec:hoefsmit-scott} give rise to a ring structure on the direct sum of the centers $Z(\bb{Q}(\X) \otimes H_{S_n})$, isomorphic to the ring of symmetric functions over $\bb{Q}(\X)$.
While the result about centers is specific to the groups $S_n$, we will show that the identities \eqref{eq:lascoux} can be deduced from uniform formulas for the traces $\tau[\Sigma_{J, \pm}]$, which work for any Weyl group $W$.

Recall that Springer constructed a $W$-action on the \'etale cohomologies of the fibers of his resolution of the unipotent variety, now called \dfemph{Springer fibers}.
In \cite{trinh}, the first author used this action to construct a trace on $H_W$ valued in $\bb{Q}(\X)$-linear traces on $W$, or equivalently, a \emph{bitrace}
\begin{align} 
	\tau_G : \bb{Q}W \otimes H_W \to \bb{Q}(\X),
\end{align}
which refines the Markov traces on $H_W$ studied by Gomi \cite{gomi} and Webster--Williamson \cite{ww_markov}.
These, in turn, were motivated by the traces used by Ocneanu to construct the HOMFLYPT link polynomial \cite{jones-v}.

In this paper, we use a normalization for $\tau_G$ characterized by the formula
\begin{align}\label{eq:tau-g-intro}
\tau_G(z \otimes T_w)|_{\X \to q}
	= \frac{1}{|G|}
		\sum_{\substack{u \in G \\ \text{unipotent}}}
		|O(w)_u| \chi_u(z)
	\quad\text{for all $z, w \in W$},
\end{align}
where $\chi_u$ is the total Springer character for $u$, reviewed in \S\ref{subsec:springer}, and $O(w)_u$ is the set of pairs $(hB, gB)$ such that $h^{-1}gB = BwB$ and $gB = uhB$.
Let $e_{J, -}$, \emph{resp.}\ $e_{J, +}$, denote the antisymmetrizer, \emph{resp.}\ symmetrizer, in $\bb{Q}W_J$, reviewed in \S\ref{subsec:symmetrizer}.
By combining \Cref{thm:main} with work of Borho--MacPherson \cite{bm}, we show:

\begin{thm}\label[thm]{thm:trace}
For any $J \subseteq S$, we have
\begin{align}
\tau[\Sigma_{J, \pm}] 
	&= (\X - 1)^{\ur{rk}(G)}\,
	\tau_G(e_{J, \pm} \otimes \whitearg)
\end{align}
as traces on $H_W$, where $\ur{rk}(G)$ is the rank of the maximal torus $\bb{T}$.
\end{thm}

From \cite{trinh}, there is a purely algebraic formula for $\tau_G$ involving the \dfemph{exotic Fourier transform}: a pairing introduced by Lusztig to relate the set $\Irr(W)$ of irreducible characters of $W$ to the set of (unipotent) irreducible characters of $G$.
Let 
\begin{align} 
\{-, -\} : \Irr(W) \times \Irr(W) \to \bb{Q}
\end{align}
be its ``truncation'' to $\Irr(W)$, and for all $\chi \in \Irr(W)$, let $\chi_\X \in R(H_W)$ be the Tits deformation of $\chi$.
We deduce the following formula.
For $G = \ur{GL}_n(\bb{F})$, where $\{-, -\}$ is the Kronecker delta, it recovers \eqref{eq:lascoux}, as we show in \Cref{prop:tau-to-e-h}.

\begin{cor}\label[cor]{cor:exotic}
For any $J \subseteq S$, we have
\begin{align}
\tau[\Sigma_{J, \pm}]
=
(\X - 1)^{\ur{rk}(G)}
\sum_{\chi, \psi \in \Irr(W)}
	\frac{\{\chi, \psi\} \psi(e_{J, \pm})}{\det(\X - e_{J, \pm} \mid \sf{V}_G)}\,\chi_\X,
\end{align}
where $\sf{V}_G$ is the reflection representation of $W$ arising from the cocharacters of $\bb{T}$.
\end{cor}

\subsection{Deodhar Decompositions}\label{subsec:cell-intro}

Our second application of \Cref{thm:main} is to provide point-counting formulas for new kinds of braid varieties through Deodhar-like decompositions.
In what follows, $h\bb{B} \xrightarrow{w} g\bb{B}$ will mean $\bb{B}h^{-1}g\bb{B} = \bb{B}w\bb{B}$.

Let $\vec{s} = (s^{(1)}, s^{(2)}, \ldots, s^{(\ell)})$ be a word in $S$.
Recall that in \cite{deodhar}, Deodhar showed how to decompose a certain \dfemph{Richardson variety} for $\vec{s}$ into subvarieties of the form $\bb{A}^\mathbf{d} \times \bb{G}_m^\mathbf{e}$, now called \dfemph{Deodhar cells}.
As in \cite{gltw},\footnote{Note that we implicitly fix a typo in \cite[(1.7)]{gltw}.} we will work with a definition depending on an element $v \in W$:
\begin{align}
	\bb{R}^{(v)}(\vec{s})
	= \{\vec{g}\bb{B} = (g_i\bb{B})_i \in (\bb{G}/\bb{B})^\ell \mid vw_\circ \bb{B} \xrightarrow{s^{(1)}} g_1\bb{B} \xrightarrow{s^{(2)}} \cdots \xrightarrow{s^{(\ell)}} g_\ell\bb{B} \xleftarrow{vw_\circ} \bb{B}\}.
\end{align}
To describe the cell decomposition, recall that a \dfemph{subword} of $\vec{s}$ is a sequence $\vec{\omega}$ of the same length $\ell$ with $\omega^{(i)} \in \{e, s^{(i)}\}$ for all $i$.
We set $\omega_{(i)} \vcentcolon= \omega^{(1)} \cdots \omega^{(i)}$.
For any $v \in W$, a \dfemph{$v$-distinguished subword} of $\vec{s}$ is a subword $\vec{\omega}$ such that 
\begin{align} 
	v\omega_{(i)}  \leq v\omega_{(i - 1)} s^{(i)}
	\quad\text{for all $i$}.
\end{align}
Let $\mathcal{D}^{(v)}(\vec{s})$ be the set of $v$-distinguished subwords $\vec{\omega}$ for which $\smash{\omega_{(\ell)}} = e$.
Then the Deodhar cells of $\bb{R}^{(v)}(\vec{s})$ are indexed by $\mathcal{D}^{(v)}(\vec{s})$.
The cell for a given element $\vec{\omega}$ is isomorphic to $\bb{A}^{\sf{d}_{\vec{\omega}}} \times \bb{G}_m^{\sf{e}_{\vec{\omega}}}$ for certain disjoint subsets $\sf{d}_{\vec{\omega}}, \sf{e}_{\vec{\omega}} \subseteq \{1, 2, \ldots, \ell\}$, allowing us to count $R^{(v)}(\vec{s}) \vcentcolon= \bb{R}^{(v)}(\vec{s})^F$:
\begin{align}\label{eq:cell-r-s}
|R^{(v)}(\vec{s})| 
	= \sum_{\vec{\omega} \in \mathcal{D}^{(v)}(\vec{s})}
	q^{|\sf{d}_{\vec{\omega}}|} 
	(q - 1)^{|\sf{e}_{\vec{\omega}}|}.
\end{align}
We give further detail in \Cref{sec:cell}.

Using \Cref{thm:main}, we relate the disjoint union of the sets $R^{(v)}(\vec{s})$ for $v \in W^{J, \mp}$ to the set $Z_J^\pm(\vec{s}) \vcentcolon= \bb{Z}_J^\pm(\vec{s})^F$ for a certain \dfemph{partial braid Steinberg variety}
\begin{align}
\bb{Z}_J^\pm(\vec{s})
	= \{(\vec{g}\bb{B}, u, y\bb{P}_J) \in (\bb{G}/\bb{B})^\ell \times \bb{Spr}_J^\pm \mid 
	u^{-1}g_\ell\bb{B} \xrightarrow{s^{(1)}} g_1\bb{B} \xrightarrow{s^{(2)}} \cdots \xrightarrow{s^{(\ell)}}  g_\ell\bb{B}
	\}.
\end{align}
We obtain identities of point counts.

\begin{thm}\label[thm]{thm:cell}
For any word $\vec{s}$, we have
\begin{align}
\frac{|Z_J^-(\vec{s})|}{|G|}
	&=
	\frac{1}{q^{\ell_J} (q - 1)^{\ur{rk}(G)}}
	\sum_{v \in W^{J, +}} 
	\sum_{\vec{\omega} \in \mathcal{D}^{(v)}(\vec{s})}
	q^{|\sf{d}_{\vec{\omega}}|} 
	(q - 1)^{|\sf{e}_{\vec{\omega}}|},\\
\frac{|Z_J^+(\vec{s})|}{|G|}
	&= 
	\frac{1}{(q - 1)^{\ur{rk}(G)}}
	\sum_{v \in W^{J, -}} 
	\sum_{\vec{\omega} \in \mathcal{D}^{(v)}(\vec{s})}
	q^{|\sf{d}_{\vec{\omega}}|} 
	(q - 1)^{|\sf{e}_{\vec{\omega}}|}.
\end{align}
(Note the sign flip between the left and right sides of each identity.)
\end{thm}

Note that $\bb{Z}_\emptyset^+(\vec{s})$ and $\bb{Z}_\emptyset^-(\vec{s})$ coincide:
They match the \dfemph{braid Steinberg variety} introduced in \cite{trinh}.
At the other extreme, $\bb{Z}_S^+(\vec{s})$ and $\bb{Z}_S^-(\vec{s})$ are the varieties respectively denoted $\cal{U}(\vec{s})$ and $\cal{X}(\vec{s})$ in \cite{trinh}.

For general $G$ and $J$, we give a decomposition of $\bb{Z}_J^-(\vec{s})$ into subvarieties that become equivalent to Deodhar cells under change of structure group:
See \Cref{cor:cohomology}.
For $G = \PGL_n(\bb{F})$, the variety $\cal{X}(\vec{s})$ was previously studied by Shende--Treumann--Zaslow \cite{stz}, who used contact geometry to construct what they call a ruling decomposition of $\cal{X}(\vec{s})$.
The recent work \cite{achllw} essentially shows that their ruling decomposition becomes the Deodhar decomposition under change of structure group, and hence, matches our own decomposition:
See \Cref{rem:stz}.

The decomposition of partial Steinberg varieties into Richardson varieties in our proof of \Cref{thm:cell} also resembles a decomposition appearing in Mellit's work on character varieties with semisimple monodromy conditions \cite{mellit}, as we discuss in \Cref{rem:mellit}.

\subsection{Combinatorics}\label{subsec:parking-intro}

Our third application of \Cref{thm:main}, by way of \Cref{thm:trace}, is to construct noncrossing sets of interest in the Catalan combinatorics of $(W, S)$.
\emph{In the rest of this introduction, $W$ is irreducible with Coxeter number $h$ and rank $r \vcentcolon= |S|$.}

Let $d_1, \ldots, d_r$ be the fundamental degrees of the action of $W$ on its (irreducible) reflection representation.
For each $i$, let $e_i = d_i - 1$.
For any positive integer $p$ coprime to $h$, the \dfemph{rational Catalan number} of $(W, p)$ is
\begin{align}
	\Cat_{W, p}
	\vcentcolon= \prod_i \frac{p + e_i}{d_i},
\end{align}
while the \dfemph{rational parking number} of $(W, p)$ is $p^r$.
These numbers enumerate disparate families of combinatorial objects, which fall into two kinds.
\emph{Nonnesting} families are constructed from root-theoretic data that generalize nonnesting partitions, \emph{resp.}\ parking functions.
In \cite{gltw}, we gave the first construction of a rational \emph{noncrossing} Catalan, \emph{resp.}\ parking, family for any finite Coxeter group $W$ and $p > 0$ coprime to $h$, recovering earlier constructions for $p = h + 1$.
These noncrossing objects differ from the nonnesting objects in that they depend on a chosen ordering of $S$, or \dfemph{Coxeter word}: a Coxeter-theoretic, not root-theoretic, datum.

For any word $\vec{s}$ in $S$, let $\cal{M}^{(v)}(\vec{s}) \subseteq \cal{D}^{(v)}(\vec{s})$ be the subset of elements $\vec{\omega}$ such that $|\mathbf{e}_{\vec{\omega}}| = r$, the minimum possible value \cite[Cor.\ 4.9]{gltw}.
Let $\vec{c}$ be a Coxeter word for $(W, S)$, and $\vec{c}^p$ its $p$-fold concatenation.
The main results of \cite{gltw} are the identities
\begin{align}
	\Cat_{W, p} = |\cal{M}^{(e)}(\vec{c}^p)|
	\quad\text{and}\quad
	p^r = \sum_{v \in W} |\cal{M}^{(v)}(\vec{c}^p)|,
\end{align}
proved by way of $\X$-deformed identities involving $\cal{D}^{(v)}(\vec{c}^p)$ and taking $\X \to 1$.

In \Cref{sec:parking}, we prove an identity that interpolates between the two above.
Let $d_1^J, \ldots, d_{|J|}^J$ be the fundamental degrees of $W_J$.
Let $e_1^J,\ldots, e_{|J|}^J$ be the exponents of the $W_J$-action on the reflection representation of $W$.
We define the \dfemph{rational parabolic parking numbers} of $(W, p, J)$ to be
\begin{align}
	\Park_{W, p}^{J, \pm}
	= \prod_i \frac{p \pm e_i^J}{d_i^J}.
\end{align}
Then $\Park_{W, p}^{S, +} = \Cat_{W, p}$ and $\Park_{W, p}^{\emptyset, +} = \Park_{W, p}^{\emptyset, -} = p^r$.
We relate these numbers to $\tau_G$ via a result from \cite{trinh}, which describes $\tau_G(\whitearg \otimes T_{\vec{c}}^p)$ for a certain $T_{\vec{c}} \in H_W$ as the graded character of a \dfemph{rational parking space} for $(W, p)$, in the sense of \cite{arr} and \cite{alw}.
Ultimately, we obtain:

\begin{cor}\label[cor]{cor:parking}
For any Coxeter word $\vec{c}$, integer $p > 0$ coprime to $h$, and subset $J \subseteq S$, we have
\begin{align}
\Park_{W, p}^{J, \pm}
	= \sum_{v \in W^{J, \mp}} |\cal{M}^{(v)}(\vec{c}^p)|.
\end{align}
(Note the sign flip.)
That is, $\coprod_{v \in W^{J, \pm}} \cal{M}^{(v)}(\vec{c}^p)$ is a $\vec{c}$-noncrossing set enumerated by the $\mp$-rational parabolic parking number of $(W, p, J)$.
\end{cor}

\subsection{Markov Traces and Kirkman Numbers}\label{subsec:markov-intro}

In \Cref{sec:markov}, we prove results about Markov traces and rational Kirkman numbers in type $A$ that are respectively parallel to \Cref{thm:trace} and \Cref{cor:parking}.
In \Cref{sec:associahedron}, we explain how our noncrossing objects for Kirkman numbers can be related to the classical combinatorics of associahedra.

First, for any $v \in W$, recall the \dfemph{left ascent set} $\Asc(v) = \{s \in S \mid \ell(sv) > \ell(v)\}$ and \dfemph{descent set} $\Des(v) = \{s \in S \mid \ell(sv) < \ell(v)\}$.
Observe that $W^{J, -}$, \emph{resp.}\ $W^{J, +}$, consists of those $v$ such that $\Asc(v) \supseteq J$, \emph{resp.}\ $\Des(v) \supseteq J$.
Hence, $N_J^S(1)$ and $\X^{-\ell_J} N_J^S(T_{w_{J\circ}}^2)$ decompose as sums, over supersets $I \supseteq J$, of elements
\begin{align}
	\zeta_I^- \vcentcolon= \sum_{\Asc(v) = I} \X^{-\ell(v)} T_{v^{-1}} T_{v\vphantom{^{-1}}}
	\quad\text{and}\quad
	\zeta_I^+ \vcentcolon= \sum_{\Des(v) = I} \X^{-\ell(v)} T_{v^{-1}} T_{v\vphantom{^{-1}}}.
\end{align}
Note that $\zeta_S^- = \zeta_\emptyset^+ = 1$ and $\zeta_\emptyset^- = \zeta_S^+ = \X^{-\ell_S} T_{w_\circ}^2$.
By inclusion-exclusion, the elements $\zeta_I^\pm$ are again central in $H_W$.

\begin{quest}
	For general $W$ and $I$, is there a more familiar description of the traces on $H_W$ of the form $\tau[\zeta_I^\pm]$?
\end{quest}

We now take $W = S_n$ and $S = \{s_1, \ldots, s_{n - 1}\}$.
The HOMFLYPT Markov trace on $H_{S_n}$ can be written as a $\bb{Q}(\X^{1/2})[a^{\pm 1}]$-valued trace $\mu_n$.
For $0 \leq k \leq n - 1$, let $\mu_n^{(k)} : H_W \to \bb{Q}(\X^{1/2})$ be the coefficient of the $k$th highest power of $a$ in $\mu_n$, and let
\begin{align}
I_k = \{s_1, s_2, \ldots, s_{n - 1 - k}\}.
\end{align}
The following result simplifies a formula for $\mu_n^{(k)}$ due to Bezrukavnikov--Tolmachov, which used symmetric polynomials in Jucys--Murphy braids \cite[Cor.\ 6.1.2]{bt}.
It would be interesting to generalize it to other types, just as \cite{tz} generalizes \emph{loc.\ cit.}:
See \S\ref{subsec:other-types}.

\begin{thm}\label[thm]{thm:markov}
For any integer $k$, we have
\begin{align}\label{eq:bt}
\tau[\zeta_{I_k}^+]
	= (\X - 1)^{n - 1} \mu_n^{(k)}
\end{align}
as traces on $H_{S_n}$.
\end{thm} 

Let $e_{\wedge^k} \in \bb{Q}W$ be the Young symmetrizer of the hook partition $(n - k, 1, \ldots, 1) \vdash n$, which indexes the $k$th exterior power of the reflection representation of $S_n$.
By combining \eqref{eq:bt} with the result in \cite{trinh} relating the Markov trace to $\tau_G$, we deduce this analogue of \Cref{thm:trace}:

\begin{cor}
For $G$ split semisimple of type $A_{n - 1}$, and any integer $k$, we have
\begin{align}
	\tau[\zeta_{I_k}^+]
	= (\X - 1)^{n - 1}\, \tau_G(e_{\wedge^k} \otimes \whitearg)
\end{align}
as traces on $H_{S_n}$.
\end{cor}

For general $W$ and $0 \leq k \leq r$, we use the rational parking space for $(W, p)$ mentioned earlier to define numbers $\Kirk_{W, p}^{(k)}$ that unify the type-$A$ rational Kirkman numbers in \cite{arw} and the Kirkman numbers for Coxeter groups in \cite{arr}.
For $W = S_n$, the preceding result implies this analogue of \Cref{cor:parking}:

\begin{cor}\label[cor]{cor:kirkman}
Take $W = S_n$ and $S = \{s_1, \ldots, s_{n - 1}\}$.
Then for any Coxeter word $\vec{c}$, any integer $p > 0$ coprime to $n$ and integer $k$, we have
\begin{align}
\Kirk_{S_n, p}^{(k)}
	= \sum_{\Asc(v) = I_k}
	|\cal{M}^{(v)}(\vec{c}^p)|.
\end{align}
That is, $\coprod_{\Asc(v) = I_k} \cal{M}^{(v)}(\vec{c}^p)$ is a $\vec{c}$-noncrossing set enumerated by the $k$th \dfemph{rational Kirkman number} of $(S_n, p)$.
\end{cor} 

In \Cref{sec:associahedron}, we relate $\coprod_{\Asc(v) = I_k} \cal{M}^{(v)}(\vec{c}^{n + 1})$ to a classical noncrossing set for the $k$th Kirkman number of $(S_n, n + 1)$: the collection of $k$-faces in the corresponding associahedron.

\subsection{Acknowledgments}

We thank Pavel Galashin, Thomas Lam, and Ian L\^{e} for helpful discussions.
During part of our work, MT was supported by an NSF Mathematical Sciences Research Fellowship, Award DMS-2002238, and NW was supported by an NSF standard grant, Award DMS-2246877.

\section{The Geometric Hecke Algebra}\label{sec:hecke}

\subsection{}

In this section, we review the general definition of the convolution algebra $H_B^G$ without assuming $G$ to be split, following \cite[\S{3.3}]{carter_95}.
At the end, we explain how to adapt $N_J^S$ to this generality.
Along the way, we review Bruhat decomposition and related facts about Borel subgroups.
Our geometric setup is essentially Kawanaka's in \cite{kawanaka}.
See also \cite{carter, carter_95}.

We keep $\bb{F}$, $q$, $\bb{G}$, $\bb{B}$, $\bb{T}$, $\bb{W}$ as in \S\ref{subsec:setup}.
Thus the characteristic of $\bb{F}$ is a good prime for $\bb{G}$ in the sense of \cite[28]{carter}.
Let $S_\bb{B}$ be the system of simple reflections of $\bb{W}$ arising from $\bb{B}$.
Let $\ell_\bb{B}$ be the Bruhat length function on $\bb{W}$ defined by $S_\bb{B}$.

\subsection{Bruhat Decomposition}\label{subsec:bruhat}

Note that $w\bb{B}$ and $\bb{B}w$ are well-defined for any $w \in \bb{W}$. 
Bruhat decomposition says that as we run over all $w$, the double cosets $\bb{B}w\bb{B}$ are pairwise disjoint and partition $\bb{G}$.

Let $\bb{U}$ be the unipotent radical of $\bb{B}$, so that $\bb{B} = \bb{T} \ltimes \bb{U}$.
Let $\bb{U}_-$ be the unipotent radical of the opposed Borel $\bb{B}_-$.
Note that $w\bb{U}w^{-1}$ and $w\bb{U}_-w^{-1}$ are well-defined for all $w \in \bb{W}$.
Let
\begin{align}
\bb{U}_w &= \bb{U} \cap w\bb{U}w^{-1},\\
\bb{U}_w^- &= \bb{U} \cap w\bb{U}_-w^{-1}.
\end{align}
Then $\bb{U}_w, \bb{U}_w^-$ are stable under the conjugation action of $\bb{T}$ on $\bb{U}$.
The following results are proved in \cite[\S{2.5}]{carter}:

\begin{lem}\label[lem]{lem:bruhat}
For all $w \in \bb{W}$:
\begin{enumerate}
\item 	
If $\ell_\bb{B}(wv) = \ell_\bb{B}(w) + \ell_\bb{B}(v)$, then $\bb{U}_{wv}^- = \bb{U}_w^- \bb{U}_v^-$, and $\bb{U}_w^- \cap \bb{U}_v^- = \{1\}$.

\item 	
$\bb{U} = \bb{U}_w \bb{U}_w^- = \bb{U}_w^- \bb{U}_w$, and $\bb{U}_w \cap \bb{U}_w^- = \{1\}$.

\item 
$\bb{B}w\bb{B} = \bb{U}_w^-w\bb{B}$, and the map $\bb{U}_w^- \to \bb{U}_w^-w\bb{B}/\bb{B}$ is an isomorphism.

\item 	
As an algebraic variety (but not group), $\bb{U}_w^-$ is the product of the root subgroups inverted by $w$, hence an affine space of dimension $\ell_\bb{B}(w)$.
	
\end{enumerate}
\end{lem}

\subsection{Bott--Samelson Varieties}

The double cosets of $\bb{B}$ in $\bb{G}$ are in bijection with the set of diagonal $\bb{G}$-orbits on $(\bb{G}/\bb{B})^2$.
As in the introduction, we write $h\bb{B} \xrightarrow{w} g\bb{B}$ to mean $\bb{B}h^{-1}g\bb{B} = \bb{B}w\bb{B}$.
Such pairs $(h\bb{B}, g\bb{B})$ form the points of the $\bb{G}$-orbit of $(\bb{G}/\bb{B})^2$ corresponding to $w$, which we will denote by $\bb{O}(w)$.
On points,
\begin{align}
\bb{O}(w)
	= \{(h\bb{B}, hw\bb{B})\}.
\end{align}
More generally, for any sequence of elements $\vec{w} = (w^{(1)}, w^{(2)}, \ldots, w^{(k)})$ in $\bb{W}$, let $\bb{O}(\vec{w})$ be the subvariety of $(\bb{G}/\bb{B})^{1 + k}$ defined on points by
\begin{align}
	\bb{O}(\vec{w}) = \{\vec{g}\bb{B} = (g_0\bb{B}, g_1\bb{B}, \ldots, g_k\bb{B}) \mid g_0\bb{B} \xrightarrow{w^{(1)}} g_1\bb{B} \xrightarrow{w^{(2)}} \cdots \xrightarrow{w^{(k)}} g_m\bb{B}\}.
\end{align}
The Zariski closure of $\bb{O}(\vec{w})$ is called the \dfemph{Bott--Samelson variety} of $\vec{w}$.
For this reason, $\bb{O}(\vec{w})$ may be called the \dfemph{open Bott--Samelson variety}.

For any subset $I \subseteq \{1, \ldots, k\}$, we write $\pr_I : \bb{O}(\vec{w}) \to (\bb{G}/\bb{B})^I$ to denote the map that sends $\pr_I(\vec{g}\bb{B}) = (g_i\bb{B})_{i \in I}$.
When writing out $\vec{w}$, \emph{resp.}\ $I$, explicitly, we will omit the parentheses, \emph{resp.}\ brackets, where convenient.

\Cref{lem:bruhat} implies that if $\ell_\bb{B}(wv) = \ell_\bb{B}(w) + \ell_\bb{B}(v)$, then $\pr_{0, 2}$ induces an explicit isomorphism $\bb{O}(w, v) \xrightarrow{\sim} \bb{O}(wv)$.
By induction, any variety of the form $\bb{O}(\vec{w})$ is explicitly isomorphic to one of the form $\bb{O}(\vec{s})$, where $\vec{s}$ is a word in $S_\bb{B}$.

\subsection{Frobenius Maps}\label{subsec:f-fixed}

For algebraic varieties over $\bar{\bb{F}}$ equipped with Frobenius maps, we use italics to denote the corresponding sets of Frobenius-fixed points.

As in \S\ref{subsec:setup}, we fix a Frobenius map $F : \bb{G} \to \bb{G}$ arising from an $\bb{F}$-form, such that $\bb{B}$ and $\bb{T}$ are $F$-stable.
Then $\bb{W}$ and $S_\bb{B}$ are also $F$-stable.
The group $W \vcentcolon= \bb{W}^F$ is again a Coxeter group, which can be identified with $N_G(T)/T$.

\begin{rem}
When $\bb{G}$ is almost-simple, the options for $G$ and $W$ are listed in \cite[\S{1.5--1.6}]{carter_95}.
Notably, $W$ is crystallographic except when it has factors of type ${}^2F_4$.
\end{rem}

There is a system of simple reflections for $W$, which we will denote $S$, indexed by the $F$-orbits on $S_\bb{B}$:
Each element $s \in S$ is the product of all the elements in the given $F$-orbit, which pairwise commute and form a reduced word in $S_\bb{B}$ in any order.
Let $\ell$ be the Bruhat length function on $W$ defined by $S$.

By Lang's theorem \cite[32]{carter}, $g\bb{B}$ is $F$-stable if and only if $g \in G$, and in this case, $gB = (g\bb{B})^F$.
Similarly, $\bb{B}w\bb{B}$ is $F$-stable if and only if $w \in W$, and in this case, $BwB = (\bb{B}w\bb{B})^F$.
Thus, the double cosets $BwB$ for $w \in W$ partition $G$, while the $G$-orbits on $(G/B)^2$ are the sets $O(w)$ for $w \in W$.
As explained in \cite{carter}, parts (1)--(3) of \Cref{lem:bruhat} have analogues with $\bb{W}$ replaced by $W$.
See also \cite[\S{1}]{kawanaka}.

\begin{lem}\label[lem]{lem:bruhat-f}
For all $w \in W$:
\begin{enumerate}
\item 	
If $\ell(wv) = \ell(w) + \ell(v)$, then $U_{wv}^- = U_w^- U_v^-$, and $U_w^- \cap U_v^- = \{1\}$.

\item 	
$U = U_w U_w^- = U_w^- U_w$, and $U_w \cap U_w^- = \{1\}$.
	
\item 
$BwB = U_w^-wB$, and the map $U_w^- \to U_w^-wB/B$ is a bijection.
	
\end{enumerate}
\end{lem}

The one point where caution is needed concerns the sizes of $U_w$ and $U_w^-$, as they involve $\ell_\bb{B}(w)$, not $\ell(w)$ \cite[74]{carter}:

\begin{lem}\label[lem]{lem:rosenlicht}
For all $w \in W$, we have $|U_w^-| = q^{\dim(\bb{U}_w^-)} = q^{\ell_\bb{B}(w)}$.
\end{lem}

\begin{ex}\label[ex]{ex:gu-3}
Take $\bb{G} = \mathbf{GL}_n$, so that the absolute Weyl group is $\bb{W} = S_n$.
Take $F(g) = (g^\tau)^{-q} = (g^{-q})^\tau$, where $(-)^\tau$ is the ``anti-transpose'' given by $g^\tau = Jg^t J$, where $J$ is the matrix with $1$’s along the anti-diagonal and $0$’s everywhere else.
Then $G = \mathrm{GU}_n(q)$.

Take $n = 3$, so that $\bb{W}$ is generated by simple reflections $s_1$ and $s_2$, which are swapped by Frobenius.
Then the relative Weyl group is $W = \bb{W}^F = \{e, w_\circ\}$, where $w_\circ = s_1s_2s_1$.
The set of relative simple reflections is $S = \{w_\circ\}$.
Observe the discrepancy between the two length functions:
\begin{align}
	\ell(w_\circ) = 1 \quad\text{but}\quad \ell_\bb{B}(w_\circ) = 3.
\end{align}
By \Cref{lem:rosenlicht}, the size of the unipotent radical is given by the latter function:
$|U_{w_\circ}^-| = q^{\ell_\bb{B}(w_\circ)} = q^3$. 
\end{ex}

\subsection{Operations on Functions}\label{subsec:functions}

For any finite set $X$ equipped with the action of a finite group $G$, we write $\cal{C}_G(X)$ to denote the free module of $\bb{Z}$-valued, $G$-invariant functions on $X$.
For any $G$-stable subset $Y \subseteq X$, we write $\sf{1}_Y \in \cal{C}_G(X)$ to denote the indicator function on $Y$.

For a $G$-equivariant map $f : Y \to X$, the \dfemph{pullback} of functions along $f$ is the linear map $f^\ast : \cal{C}_G(X) \to \cal{C}_G(Y)$ given by $f^\ast(\varphi)(y) = \varphi(f(y))$.
The \dfemph{pushforward}, or \dfemph{integral}, of functions along $f$ is the linear map $f_! : \cal{C}_G(Y) \to \cal{C}_G(X)$ given by
\begin{align} 
	f_!(\psi)(x) = \sum_{y \in f^{-1}(x)} \psi(y).
\end{align} 
When $f$ can be understood from context, we omit $f_!$ from our notation.

Let $\ast$ denote the \dfemph{convolution product} on $\cal{C}_G(X \times X)$ defined in terms of the three projection maps $\pr_{i,j} : X^3 \to X^2$ by \begin{align} 
\varphi_1 \ast \varphi_2 = \pr_{1,3,!}(\pr_{1,2}^\ast(\varphi_1) \cdot \pr_{2,3}^\ast(\varphi_2)),
\end{align}
where $\cdot$ denotes pointwise multiplication.
Explicitly, 
\begin{align}
	(\varphi_1 \ast \varphi_2)(y, x) = \sum_{z \in X} \varphi_1(y, z)\varphi_2(z, x).
\end{align}
The indicator function of the diagonal $\{(x, x) \mid x \in X\} \subseteq X^2$ is the identity element for this operation.
If $X$ is equipped with a $G$-action, and $G$ acts on $X^2$ diagonally, then $\ast$ restricts to an operation on $\cal{C}_G(X \times X)$ with the same identity element.

Iwahori proved that the ring formed by $\cal{C}_G(G/B \times G/B)$ under convolution is freely generated by the elements $\sf{1}_w \vcentcolon= \sf{1}_{O(w)}$ for $w \in W$ modulo the following relations for all $w \in W$ and $s \in S$:
\begin{align}
\sf{1}_s \ast \sf{1}_w = \left\{\begin{array}{ll}
	\sf{1}_{sw}
	&\ell(sw) > \ell(w),\\
	|U_s^-|\, \sf{1}_{sw} + (|U_s^-| - 1)\, \sf{1}_w
	&\ell(sw) < \ell(w).
\end{array}\right.
\end{align}
See \cite[\S{3.3}]{carter_95} or \cite[Thm.\ 2.6]{kawanaka}.
We define $H_B^G$ to be the $\bb{Z}[\tfrac{1}{q}]$-algebra
\begin{align}
H_B^G = \cal{C}_G(G/B \times G/B)[\tfrac{1}{q}].
\end{align}
If $G$ is \dfemph{split}, meaning $W = \bb{W}$, then $\ell_\bb{B}(s) = \ell(s) = 1$ and $|U_s^-| = q$ for all $s \in S$.
This is the case on which the introduction focused.
Here, $W$ is crystallographic, and $H_B^G$ is a specialization of the $\bb{Z}[\X^{\pm 1}]$-algebra $H_W$ freely generated by elements $T_w$ for $w \in W$ modulo the following relations for all $w \in W$ and $s \in S$:
\begin{align}
T_s T_w = \left\{\begin{array}{ll}
T_{sw}
	&\ell(sw) = \ell(w) + 1,\\
\X T_{sw} + (\X - 1)T_w
	&\ell(sw) = \ell(w) - 1.
\end{array}\right.
\end{align}

\subsection{Parabolic Subgroups}

Fix an $F$-stable subset $J_\bb{B} \subseteq S_\bb{B}$, corresponding to a subset $J \subseteq S$.
Let $\bb{W}_J \subseteq \bb{W}$, \emph{resp.}\ $W_J \subseteq W$, be the subgroup generated by $J_\bb{B}$, \emph{resp.}\ $J$.
Then $\bb{W}_J$ is $F$-stable and $W_J = \bb{W}_J^F$.

Let $\bb{P}_J = \bb{B}\bb{W}_J\bb{B} \subseteq \bb{G}$.
We can write $\bb{P}_J = \bb{L}_J \ltimes \bb{U}_J$, where $\bb{L}_J$ is reductive with Weyl group $\bb{W}_J$ and $\bb{U}_J$ the unipotent radical of $\bb{P}_J$.
These subgroups are $F$-stable, and on $F$-fixed points, we have $P_J = L_J \ltimes U_J$.

By construction, $\bb{B}_J = \bb{L}_J \cap \bb{B}$ is a Borel subgroup of $\bb{L}_J$.
The inclusion $L_J \subseteq P_J$ descends to an $L_J$-equivariant bijection $L_J/B_J \simeq P_J/B$, which in turn yields an isomorphism of algebras
\begin{align}
\cal{C}_{L_J}(L_J/B_J \times L_J/B_J)
	\simeq \cal{C}_{L_J}(P_J/B \times P_J/B).
\end{align}
Once we adjoin $\tfrac{1}{q}$, the left-hand side becomes $H_{B_J}^{L_J}$, and the right-hand side becomes the subalgebra of $H_B^G$ generated by the elements $\sf{1}_w$ with $w \in W_J$.
Henceforth, we identify these $\bb{Z}[\tfrac{1}{q}]$-algebras with each other.

As in the introduction, let $W^{J, -} \subseteq W$ be the set of minimal-length right coset representatives for $W_J$.
By \Cref{lem:bruhat}(4) and \Cref{lem:rosenlicht}, the split case of the definition below recovers the $\X \to q$ specialization of the relative norm map in \S\ref{subsec:hoefsmit-scott}.

\begin{df}\label[df]{df:norm}
The \dfemph{relative norm} map $N_J^S : H_{B_J}^{L_J} \to H_B^G$ is defined by
\begin{align}
N_J^S(\alpha) = \sum_{v \in W^{J, -}}
	\frac{1}{|U_v^-|}\,
	\sf{1}_{v^{-1}} \ast \alpha \ast \sf{1}_{v\vphantom{^{-1}}}.
\end{align}
We have implicitly used \Cref{lem:rosenlicht} to ensure that $|U_v^-|$ is a power of $q$.
\end{df}

\begin{ex}
We continue \Cref{ex:gu-3}, where $\bb{G} = \mathbf{GL}_3$ and $G = \mathrm{GU}_3(q)$.
The Hecke algebra $H_B^G$ is defined by the relation
\begin{align}
	\sf{1}_{w_\circ} \ast \sf{1}_{w_\circ} = q^3 \sf{1}_e + (q^3 - 1) \sf{1}_{w_\circ}.
\end{align}
We now compute $N_\emptyset^S(1)$ by hand.
Note that $\bb{P}_\emptyset = \bb{B}$ and $\bb{L}_\emptyset = \bb{T}$, from which $W^{\emptyset, -} = W$.
Using the calculation $|U_{w_\circ}^-| = q^3$ from \Cref{ex:gu-3} and the Hecke relation above, we get
\begin{align}
	N_\emptyset^S(1) &= \sum_{v \in \{e, w_\circ\}} \frac{1}{|U_v^-|}\, \sf{1}_{v^{-1}} \ast \sf{1}_{v\vphantom{^{-1}}}\\
	&= \sf{1}_e + q^{-3}\, \sf{1}_{w_\circ}^2\\
	&= 2\,\sf{1}_e + (1 - q^{-3}) \sf{1}_{w_\circ}.
\end{align}
Looking ahead to \Cref{sec:springer}, one can verify \eqref{eq:main} by checking that the last expression matches $\mult_! \sf{1}_{E_\emptyset^-}$.
Note that $E_\emptyset^- = E_\emptyset^+$, since $\bb{Spr}_\emptyset^- = \bb{Spr}_\emptyset^+$.
\end{ex}

\subsection{A Lemma on Unipotent Subgroups}\label{subsec:longest}

Let $w_\circ$ and $w_{J\circ}$ respectively denote the longest elements of $W$ and $W_J$ with respect to $S$.
Then $U = U_{w_\circ}$ and $U_J = U_{w_{J\circ}}$.
The following fact will be useful:

\begin{lem}\label[lem]{lem:j}
For any $J \subseteq S$ and $v \in W^{J, -}$, we have
\begin{align}
U_J \cap U_v &= U_{w_{J\circ}v},\\
U_J \cap U_v^- &= U_v^-.
\end{align}
In particular, $U_J = U_{w_{J\circ}v}U_v^- = U_v^-U_{w_{J\circ}v}$ and $U_{w_{J\circ}v} \cap U_v^- = \{1\}$.
In the split case, the analogous identities hold with $\bb{U}_J$, $\bb{U}_v$, \emph{etc.}\ in place of $U_J$, $U_v$, \emph{etc.}.
\end{lem}

\begin{proof}
To show $U_J \cap U_v = U_{w_{J\circ}v}$:
In general, if $w, v \in W$ satisfy $\ell(wv) = \ell(w) + \ell(v)$, then $U_{wv}^- = U_w^- U_v^-$ and $U_w^- \cap U_v^- = \{1\}$ by \Cref{lem:bruhat-f}(1), which implies that $U_{wv} = U_w \cap U_v$ by \Cref{lem:bruhat-f}(2).

To show $U_J \cap U_v^- = U_v^-$, meaning $U_v^- \subseteq U_J$:
In general, if $w \in W_J$ and $v \in W^{J, -}$, then the $F$-orbits of root subgroups of $\bb{U}_J$ inverted by $wv$ are precisely those inverted by $w$.
Taking $w = e$ gives the result.

In the split case, $\ell_\bb{B} = \ell$, and thus, $v$ minimizes $\ell_\bb{B}$ in $\bb{W}_Jv$.
So we can repeat all the arguments above with the varieties in place of the sets.
\end{proof}

\section{Partial Springer Resolutions}\label{sec:springer}

\subsection{}

Recall the partial Springer resolutions $\bb{Spr}_J^\pm \subseteq \bb{G} \times \bb{G}/\bb{P}_J$ and the varieties $\bb{E}_J^\pm = \bb{G}/\bb{B} \times \bb{Spr}_J^\pm$ from \S\ref{subsec:setup}.
The latter are stable under the left $\bb{G}$-action on $\bb{G}/\bb{B} \times \bb{G} \times \bb{G}/\bb{P}_J$ defined by
\begin{align}\label{eq:action}
g \cdot (h\bb{B}, u, y\bb{P}_J) = (gh\bb{B}, gug^{-1}, gy\bb{P}_J).
\end{align}
Let $\mult : \bb{G}/\bb{B} \times \bb{G} \times \bb{G}/\bb{P}_J \to (\bb{G}/\bb{B})^2$ be the $\bb{G}$-equivariant map defined by
\begin{align}
\mult(h\bb{B}, u, y\bb{P}_J) = (h\bb{B}, uh\bb{B}).
\end{align}
On $F$-fixed points, it restricts to $G$-equivariant maps $\mult : E_J^\pm \to (G/B)^2$.
These recover the maps $\mult$ in \S\ref{subsec:setup}.
The goal of this section is to prove the identities
\begin{align}\begin{split}\label{eq:main}
\mult_!\sf{1}_{E_J^-}
	&= |U_J|\, N_J^S(1),\\
\mult_!\sf{1}_{E_J^+}
	&= |U_J|\, N_J^S(\sf{1}_{w_{J\circ}}^2),
\end{split}\end{align} 
where $N_J^S$ is now given by \Cref{df:norm}.
They recover \Cref{thm:main} in the split case.

\subsection{Reduction to Strata}

Observe that $\bb{E}_J^\pm$ is a union of $\bb{G}$-stable subvarieties $\bb{E}_{J, v}^\pm$ for $\bb{W}_Jv \in \bb{W}_J\backslash \bb{W}$, where on points,
\begin{align} 
\bb{E}_{J, v}^\pm
	= \{(h\bb{B}, u, y\bb{P}_J) \in \bb{G}/\bb{B} \times \bb{Spr}_J^\pm \mid \bb{P}_Jy^{-1}h\bb{B} = \bb{P}_Jv\bb{B}\}.
\end{align}
From \S\ref{subsec:f-fixed}, we see that $\bb{P}_Jv\bb{B}$ is $F$-stable if and only if $v \in W$, and in this case, $P_JvB = (\bb{P}_Jv\bb{B})^F$.
Therefore, $E_J^\pm$ is the union of its $G$-stable subsets $E_{J, v}^\pm$ as $v$ runs over a full set of right coset representatives for $W_J$: for instance, $W^{J, -}$.
As \Cref{lem:j} shows that $U_J \simeq U_{w_{J\circ}v} \times U_v^-$, we reduce \eqref{eq:main} to:

\begin{thm}\label[thm]{thm:main-v}
If $v \in W^{J, -}$, then:
\begin{enumerate}
\item 	$\mult_!\sf{1}_{E_{J, v}^-} = |U_{w_{J\circ}v}|\, \sf{1}_{v^{-1}} \ast \sf{1}_{v\vphantom{^{-1}}}$.

\item 	$\mult_!\sf{1}_{E_{J, v}^+} = |U_{w_{J\circ}v}|\, \sf{1}_{v^{-1}} \ast \sf{1}_{w_{J\circ}}^2 \ast \sf{1}_{v\vphantom{^{-1}}}$.

\end{enumerate} 
\end{thm}

\subsection{Reduction to the Borel}\label{subsec:p-to-b}

Let $\check{\bb{E}}_{J, v}^\pm \subseteq \bb{G}/\bb{B} \times \bb{G} \times \bb{G}/\bb{B}$ be the subvariety defined on points by
\begin{align}
	\check{\bb{E}}_{J, v}^\pm
	&= \{(h\bb{B}, u, y\bb{B}) \mid 
	\text{$(u, y\bb{P}_J) \in \bb{Spr}_J^\pm$ and $y\bb{B} \xrightarrow{v} h\bb{B}$}\}.
\end{align}
The forgetful map $\bb{G}/\bb{B} \to \bb{G}/\bb{P}_J$ induces a map $\check{\bb{E}}_{J, v}^\pm \to \bb{E}_{J, v}^\pm$.

\begin{lem}\label[lem]{lem:check}
If $v \in W^{J, -}$, then $\check{E}_{J, v}^\pm \to E_{J, v}^\pm$ is a bijection.
In the split case, this bijection arises from an isomorphism $\check{\bb{E}}_{J, v}^\pm \to \bb{E}_{J, v}^\pm$.
\end{lem}

\begin{proof}
The first claim is just the fact that if $v$ minimizes $\ell$ in $W_Jv$, then there are compatible bijections from $U_v^-$ to the Schubert cells $BvB/B$ and $BvP_J/P_J$.

For the second claim:
As in the proof of \Cref{lem:j}, $v$ minimizes $\ell_\bb{B}$ in $\bb{W}_Jv$.
So we can repeat the argument above, but with the varieties $\bb{U}_v^-$, $\bb{B}$, $\bb{P}_J$ in place of the sets $U_v^-$, $B$, $P_J$, and isomorphisms in place of bijections.
\end{proof}

The varieties $\check{\bb{E}}_J^\pm$ are stable under the $\bb{G}$-action on $\bb{G}/\bb{B} \times \bb{G} \times \bb{G}/\bb{B}$ analogous to \eqref{eq:action}.
Let $\check{\mult} : \bb{G}/\bb{B} \times \bb{G} \times \bb{G}/\bb{B} \to (\bb{G}/\bb{B})^3$ be the equivariant map defined by 
\begin{align}
\check{\mult}(h\bb{B}, u, y\bb{B})
	= (h\bb{B}, y\bb{B}, uh\bb{B}).
\end{align}
The proofs of the two parts of \Cref{thm:main-v} will use $\check{\mult}$ in different ways.

\subsection{Proof of (1)}

In the notation of \Cref{sec:hecke},
\begin{align} 
\pr_{0, 2, !}\sf{1}_{O(v^{-1}, v)} 
	&= \sf{1}_{v^{-1}} \ast \sf{1}_{v\vphantom{^{-1}}}.
\end{align} 
This suggests comparing $\bb{E}_{J, v}^-$ to a bundle over $\bb{O}(v^{-1}, v)$.
It turns out that $\check{\bb{E}}_{J, v}^-$ is the bundle we seek.

Observe that if $(h\bb{B}, u, y\bb{B})$ is a point of $\check{\bb{E}}_{J, v}^-$, then $\bb{B}y^{-1}uh\bb{B} = \bb{B}y^{-1}h\bb{B} = \bb{B}v\bb{B}$.
Therefore, $\check{\mult}$ restricts to a map from $\check{\bb{E}}_{J, v}^-$ into $\bb{O}(v^{-1}, v)$, giving an equivariant commutative diagram:
\begin{equation}
\begin{tikzcd}
\check{\bb{E}}_{J, v}^- \ar{r} \ar{d}[left]{\check{\mult}}
	&\bb{E}_{J, v}^- \ar[bend left]{ddl}{\mult}\\
\bb{O}(v^{-1}, v) \ar{d}[left]{\pr_{0, 2}}\\
(\bb{G}/\bb{B})^2
\end{tikzcd}
\end{equation}
By \Cref{lem:check} and this diagram, we reduce case (1) of \Cref{thm:main-v} to:

\begin{prop}\label[prop]{prop:minus-case}
If $v \in W^{J, -}$, then
\begin{align} 
\check{\mult}_! \sf{1}_{\check{E}_{J, v}^-} = |U_{w_{J\circ}v}|\, \sf{1}_{O(v^{-1}, v)}
\end{align}
in $\cal{C}_G(O(v^{-1}, v))$.
In the split case, this identity arises from $\check{\mult} : \check{\bb{E}}_{J, v}^- \to \bb{O}(v^{-1}, v)$ being a smooth fiber bundle that restricts to a $\bb{U}_{w_{J\circ}v}$-torsor over the subvariety of $\bb{O}(v^{-1}, v)$ where $(g_0\bb{B}, g_1\bb{B}) = (v\bb{B}, \bb{B})$.
\end{prop}

\begin{proof}
For the first claim:
Recall that the $G$-action on pairs $(g_0B, g_1B) \in O(v^{-1})$ is transitive.
So by equivariance of $\check{\mult}$ and homogeneity, it suffices to compute $\check{\mult}$ over a subset of $O(v^{-1}, v)$ where these coordinates are fixed.

We take $(g_0B, g_1B) = (vB, B)$.
Over this pair, the fiber of $\check{E}_J^-$ consists of $(vB, u, B)$ with $u \in U_J$, the fiber of $O(v^{-1}, v)$ consists of $(vB, B, gB)$ with $gB \in BvB/B$, and $\check{\mult}$ is given by $u \mapsto uvB$.
Therefore, under the bijections $U_J \simeq U_{w_{J\circ}v} \times U_v^-$ of \Cref{lem:j} and $BvB/B \simeq U_v^-$ of \Cref{lem:bruhat-f}(3), $\check{\mult}$ corresponds to the projection $U_{w_{J\circ}v} \times U_v^- \to U_v^-$.
This proves the claim.

For the second claim:
As in the proof of \Cref{lem:j}, we observe that $v$ minimizes $\ell_\bb{B}$ in $\bb{W}_Jv$.
So we can repeat the arguments above with the varieties $\bb{G}$, $\bb{O}(v)$, \emph{etc.}\ in place of the sets $G$, $O(v)$, \emph{etc.}, and \Cref{lem:bruhat} in place of \Cref{lem:bruhat-f}.
\end{proof}

\subsection{Proof of (2)}

In the notation of \Cref{sec:hecke}, particularly \S\ref{subsec:longest}, 
\begin{align} 
\pr_{0, 4, !}\sf{1}_{O(v^{-1}, w_{J\circ}, w_{J\circ}, v)} 
	&=	\sf{1}_{v^{-1}} \ast \sf{1}_{w_{J\circ}}^2 \ast \sf{1}_{v\vphantom{^{-1}}}.
\end{align} 
This suggests comparing $\bb{E}_{J, v}^+$ to a bundle over $\bb{O}(v^{-1}, w_{J\circ}, w_{J\circ}, v)$.
But unlike the situation in case (1), there is no obvious map from $\check{\bb{E}}_{J, v}^+$ into the latter variety.

We do know that $\check{\mult}$ restricts to a map from $\check{\bb{E}}_{J, v}^+$ into $\bb{O}(v^{-1}) \times \bb{G}/\bb{B}$, giving an equivariant commutative diagram:
\begin{equation}	
\begin{tikzcd}
\check{\bb{E}}_{J, v}^+ \ar{r} \ar{d}[left]{\check{\mult}}
	&\bb{E}_{J, v}^+ \ar[bend left]{ddl}{\mult}\\
	\bb{O}(v^{-1}) \times \bb{G}/\bb{B} \ar{d}[left]{\pr_0 \times \id}\\
	(\bb{G}/\bb{B})^2
\end{tikzcd}
\end{equation}
At the same time, we have a map
\begin{align}
\bb{O}(v^{-1}, w_{J\circ}, w_{J\circ}, v)
	\xrightarrow{\pr_{0,1,4}} \bb{O}(v^{-1}) \times \bb{G}/\bb{B}.
\end{align}
So by \Cref{lem:check} and this discussion, we reduce case (2) of \Cref{thm:main-v} to:

\begin{prop}\label[prop]{prop:plus-case}
If $v \in W^{J, -}$, then
\begin{align}
\check{\mult}_! \sf{1}_{E_{J, v}^+}
=
	|U_{w_{J\circ}v}|\,
	\pr_{0,1,4,!}\sf{1}_{O(v^{-1}, w_{J\circ}, w_{J\circ}, v)}
\end{align}
in $\cal{C}_G(O(v^{-1}) \times G/B)$.
\end{prop}

\begin{proof}
Since the $O(w)$ partition $(G/B)^2$, it suffices to fix $w \in W$ and restrict to
\begin{align} 
O(v^{-1}) \times_w G/B = \{(hB, yB, gB) \in O(v^{-1}) \times G/B \mid Bh^{-1}gB = BwB\},
\end{align}
the preimage of $O(w)$ along $\pr_0 \times \id$.
Recall that the $G$-action on $O(w)$ is transitive.
So by equivariance and homogeneity, the fibers of $\check{E}_{J, v}^+$ and $O(v^{-1}, w_{J\circ}, w_{J\circ}, v)$ have constant size over $O(v^{-1}) \times_w G/B$.
So it suffices to compare them over a subvariety of $O(v^{-1}) \times_w G/B$ where the coordinates $(hB, gB)$ are fixed.
Moreover, to do this, it suffices to fix $hB$ and average over $gB \in hBwB/B$.

We take $hB = B$.
Then we must compare the preimages of 
\begin{align}\label{eq:b-yb-gb}
\{(B, yB, gB) \in O(v^{-1}) \times_w G/B\}
\end{align}
in $\check{E}_J^+$ and $O(v^{-1}, w_{J\circ}, w_{J\circ}, v)$.
Since $v \in W^{J, -}$, we can trade the latter set and the map $\pr_{0,1,4}$ for the set $O(v^{-1}, w_{J\circ}, w_{J\circ}v)$ and the map $\pr_{0,1,3}$.

The preimage of \eqref{eq:b-yb-gb} in $\check{E}_J^+$ consists of $(B, u, yB)$ such that $u \in yV_Jy^{-1}$ and $u \in BwB$.
Hence it has size
\begin{align}\label{eq:kawanaka-1}
|yV_Jy^{-1} \cap BwB|.
\end{align}
The preimage of \eqref{eq:b-yb-gb} in $O(v^{-1}, w_{J\circ}, w_{J\circ}v)$ consists of $(B, yB, zB, gB)$ such that
\begin{align}\label{eq:zb-gb-1}
yB \xleftarrow{w_{J\circ}} zB \xrightarrow{w_{J\circ}v} gB
\end{align} 
and $gB \in BwB/B$.
Observe that $yB \in Bv^{-1}B/B$, so homogeneity under left multiplication by $B$ lets us count the preimage for a given $yB$ by averaging over the preimages for all $yB \in Bv^{-1}B/B$.
Since $v \in W^{J, -}$, \Cref{lem:bruhat-f}(1) shows that the union of these preimages is parametrized by $(zB, gB)$ such that
\begin{align}\label{eq:zb-gb-2}
B \xleftarrow{w_{J\circ}v} zB \xrightarrow{w_{J\circ}v} gB
\end{align} 
and $gB \in BwB/B$.
It also shows that there is a bijection from $U_{(w_{J\circ}v)^{-1}}^- \times U_{w_{J\circ}v\vphantom{^{-1}}}^-$ to the set of pairs $(zB, gB)$ satisfying \eqref{eq:zb-gb-2}, given by
\begin{align} 
(u, u') \mapsto (u(w_{J\circ}v)^{-1}B, u(w_{J\circ}v)^{-1}u'w_{J\circ}vB).
\end{align}
So the set of $(zB, gB)$ satisfying \eqref{eq:zb-gb-2} and $gB \in BwB/B$ is parametrized by
\begin{align}
(U_{(w_{J\circ}v)^{-1}}^- (w_{J\circ}v)^{-1} U_{w_{J\circ}v}^-w_{J_\circ}v) \cap BwB.
\end{align}
Since $U_{(w_{J\circ}v)^{-1}}^- \subseteq B$, this last set can be identified with
\begin{align}
U_{(w_{J\circ}v)^{-1}}^- \times ((w_{J\circ}v)^{-1} U_{w_{J\circ}v}^-w_{J_\circ}v \cap BwB).
\end{align}
By \Cref{lem:bruhat-f}(3), we have $|U_{v^{-1}}^-|$ many choices for $yB \in Bv^{-1}B/B$, and since $v \in W^{J, -}$, we also have $|U_{(w_{J\circ}v)^{-1}}^-| = |U_{w_{J\circ}}^-| |U_{v^{-1}}^-|$.
Altogether, we conclude that the size of the preimage of \eqref{eq:b-yb-gb} in $O(v^{-1}, w_{J\circ}, w_{J\circ}v)$ is 
\begin{align}\label{eq:kawanaka-2}
|U_{w_{J\circ}}^-| |(w_{J\circ}v)^{-1} U_{w_{J\circ}v}^-w_{J_\circ}v \cap BwB|.
\end{align}
Finally, we compare \eqref{eq:kawanaka-1} and \eqref{eq:kawanaka-2}.
Theorem 4.1 of \cite{kawanaka} says
\begin{align}
|yV_Jy^{-1} \cap BwB| = |U_{v^{-1}}| |(w_{J\circ}v)^{-1} U_{w_{J\circ}v}^-w_{J_\circ}v \cap BwB|.
\end{align}
Again using $|U_{(w_{J\circ}v)^{-1}}^-| = |U_{w_{J\circ}}^-| |U_{v^{-1}}^-|$, we see that $|U_{v^{-1}}| = |U_{(w_{J\circ}v)^{-1}}| |U_{w_{J\circ}}^-| = |U_{w_{J\circ}v}| |U_{w_{J\circ}}^-|$, giving the desired identity.
\end{proof}

\begin{rem}
The asymmetry of the variety $\bb{O}(v^{-1}) \times \bb{G}/\bb{B}$ may seem defective.
To make the geometry more symmetrical, one might try to replace the diagram
\begin{align} 
\bb{E}_J^+ \xrightarrow{\check{\mult}} \bb{O}(v^{-1}) \times \bb{G}/\bb{B} \xleftarrow{\pr_{0,1,4}} \bb{O}(v^{-1}, w_{J\circ}, w_{J\circ}, v)
\end{align} 
with the diagram
\begin{align} 
\bb{E}_J^+ \xrightarrow{\check{\mult}'} \bb{O}(v^{-1}) \times \bb{O}(v) \xleftarrow{\pr_{0,1,3,4}} \bb{O}(v^{-1}, w_{J\circ}, w_{J\circ}, v)
\end{align} 
in which $\check{\mult}'(h\bb{B}, u, x\bb{B}) = (h\bb{B}, x\bb{B}, ux\bb{B}, uh\bb{B})$.
Then one would hope that
\begin{align}
\check{\mult}'_! \sf{1}_{E_{J, v}^+}
	=
	|U_J|\,
	\pr_{0,1,3,4,!}\sf{1}_{O(v^{-1}, w_{J\circ}, w_{J\circ}, v)}
\end{align}
in $\cal{C}_G(O(v^{-1}) \times O(v))$.
However, Kawanaka's work does not seem to establish this stronger identity.
\end{rem}

\section{Traces on the Hecke Algebra}\label{sec:trace}

\subsection{}

The goal of this section is to prove a version of \Cref{thm:trace} for general $G$, and deduce \Cref{cor:exotic} for split $G$.
We keep the general setup of \Cref{sec:hecke}.

\subsection{Traces from Relative Norms} 

As in \S\ref{subsec:trace-intro}, let $\tau : H_B^G \to \bb{Z}[\tfrac{1}{q}]$ be the trace given by $\tau(\sf{1}_e) = 1$ and $\tau(\sf{1}_w) = 0$ for all $w \neq e$, and for any central element $\zeta \in Z(H_B^G)$, let $\tau[\zeta] : H_B^G \to \bb{Z}[\tfrac{1}{q}]$ be the trace given by $\tau[\zeta](\beta) = \tau(\beta \ast \zeta)$.

\begin{lem}\label[lem]{lem:tau-norm}
	For all $J \subseteq S$ and $w \in W$ and $\alpha \in Z(H_{B_J}^{L_J})$, we have
	\begin{align}
	\frac{1}{|B|}\,\tau[N_J^S(\alpha)](\sf{1}_w) 
		&= \frac{1}{|G|}
		\sum_{(hB, gB) \in O(w)} 
		N_J^S(\iota(\alpha))(hB, gB),
	\end{align}
	where $\iota$ is the additive anti-involution of $H_{B_J}^{L_J}$ given by $\iota(\sf{1}_{w\vphantom{^{-1}}}) = \sf{1}_{w^{-1}}$.
\end{lem}

\begin{proof}
	For any $\beta \in H_B^G$ and $xB \in G/B$, we have $\tau(\beta) = \beta(xB, xB)$.
	Moreover, $|G/B| = |G|/|B|$.
	So for any $\zeta \in Z(H_B^G)$, we have
	\begin{align}
	\frac{|G|}{|B|}\, \tau[\zeta](\beta)
		&= \sum_{xB \in G/B} (\beta \ast \zeta)(xB, xB).
	\end{align}
	Next, for any $w, v, z \in W$, observe that there is a bijection
	\begin{align}\label{eq:richardson-w}
		&\{(x_0B, x_1B, x_2B, x_3B, x_4B) \in O(w, v^{-1}, z, v) \mid x_0B = x_4B\}\\
		&\qquad\xrightarrow{\sim}
		\{(g_0B, g_1B, g_2B, g_3B) \in O(v^{-1}, z^{-1}, v) \mid g_0B \xrightarrow{w} g_3B\}
	\end{align}
	given by $(g_0B, g_1B, g_2B, g_3B) = (x_4B, x_3B, x_2B, x_1B)$.
	This shows the identity
	\begin{align}
		\sum_{gB \in G/B} (\sf{1}_{w\vphantom{^{-1}}} \ast \sf{1}_{v^{-1}} \ast \sf{1}_{z\vphantom{^{-1}}} \ast \sf{1}_{v\vphantom{^{-1}}})(gB, gB)
		= \sum_{(hB, gB) \in O(w)}
		(\sf{1}_{v^{-1}} \ast \sf{1}_{z^{-1}} \ast \sf{1}_{v\vphantom{^{-1}}})(hB, gB).
	\end{align}
	By expanding $\alpha$ in the basis $(\sf{1}_z)_{z \in W_J}$ for $H_{B_J}^{L_J}$, and summing over all $v \in W^{J, -}$, we deduce that
	\begin{align}
		\sum_{xB \in G/B} (\beta \ast N_J^S(\alpha))(xB, xB)
		= \sum_{(hB, gB) \in O(w)} 
		N_J^S(\iota(\alpha))(hB, gB),
	\end{align}
	concluding the proof.
\end{proof}

\subsection{Springer Fibers}\label{subsec:springer}

A reference for this subsection is \cite{shoji}.

In order to work with \'etale cohomology, we fix a prime $\ell$ invertible in $\bb{F}$.
The notation $\ur{H}^\ast(-, \QL)$ will always mean \'etale cohomology with coefficients in the constant $\QL$-sheaf.
Henceforth, let $\bb{V} = \bb{V}_S$ and 
\begin{align} 
\bb{Spr} = \bb{Spr}_\emptyset^+ = \bb{Spr}_\emptyset^- \subseteq \bb{V} \times \bb{G}/\bb{B}.
\end{align} 
By the \dfemph{Springer resolution}, we mean either $\bb{Spr}$ or the projection map from $\bb{Spr}$ onto $\bb{V}$.
For any $u \in \bb{V}$, the \dfemph{Springer fiber} over $u$ is the (reduced) fiber of this map over $u$, viewed as a subvariety $\bb{Spr}_u$ of $\bb{G}/\bb{B}$.
On points,
\begin{align}
\bb{Spr}_u = \{y\bb{B} \in \bb{G}/\bb{B} \mid u \in y\bb{U}y^{-1}\}.
\end{align}
Springer showed that this is a projective variety with no odd cohomology.
For $u \in V \vcentcolon= \bb{V}^F$, he constructed an action of $W$ on $\ur{H}^\ast(\bb{Spr}_u)$ through a type of Fourier transform.
Later, other authors gave independent constructions, generalizing to other base fields like the complex numbers.

In this paper, we use the $W$-action on $\ur{H}^\ast(\bb{Spr}_u)$ constructed through perverse sheaf theory, which differs from Springer's original action by a sign twist.
Let $\chi_u : \bb{Q}W \to \QL$ be the trace defined by
\begin{align}
\chi_u(w) 
	= \tr(Fw \mid \ur{H}^\ast(\bb{Spr}_u)).
\end{align}
For our choice of action, the sign character of $W$ only occurs in $\chi_1$.

As reviewed in \cite[\S{15}]{shoji}, it is now known $\chi_u$ arises from the specialization at $\X \to q$ of a $\bb{Z}[\X]$-valued trace on $\bb{Z}W$.
In particular, $\chi_u(w) \in \bb{Z}$ for all $w \in W$.

\subsection{Partial Springer Fibers}\label{subsec:symmetrizer}

For all $J \subseteq S$, the \dfemph{symmetrizer} and \dfemph{antisymmetrizer} in $\bb{Q}W_J$ are respectively defined by
\begin{align}
e_{J, +} = \frac{1}{|W_J|} \sum_{w \in W_J} w
	\quad\text{and}\quad
	e_{J, -} = \frac{1}{|W_J|} \sum_{w \in W_J} (-1)^{\ell(w)} w.
\end{align} 
These are central elements of $\bb{Q}W_J$, such that $\bb{Q}W_Je_{J, +}$ and $\bb{Q}W_Je_{J, -}$ respectively afford the trivial and sign representations of $W_J$.

Borho--MacPherson related $e_{J, -}$ and $e_{J, +}$ to the \dfemph{partial Springer fibers}
\begin{align}
\bb{Spr}_{J, u}^- 
	&= \{y\bb{P}_J \in \bb{G}/\bb{P}_J \mid u \in y\bb{U}_Jy^{-1}\},\\
\bb{Spr}_{J, u}^+
	&= \{y\bb{P}_J \in \bb{G}/\bb{P}_J \mid u \in y\bb{V}_Jy^{-1}\}.
\end{align}
By \S\ref{subsec:f-fixed}, the set of $F$-fixed points $\mathit{Spr}_{J, u}^-$, \emph{resp.}\ $\mathit{Spr}_{J, u}^+$, is the set of $yP_J \in G/P_J$ such that $u \in yU_Jy^{-1}$, \emph{resp.}\ $u \in yV_Jy^{-1}$.
For our choice of Springer action, the main result of \cite{bm} implies that for all $J \subseteq S$ and $u \in V$, we have
\begin{align}\begin{split}\label{eq:bm}
\frac{1}{ |U_{w_{J\circ}}^-|}\, \chi_u(e_{J, -})
	&= |\mathit{Spr}_{J, u}^-|,\\
\chi_u(e_{J, +})
	&= |\mathit{Spr}_{J, u}^+|.
\end{split}\end{align}
More precisely, these results come from transferring Borho--MacPherson's arguments from sheaves in the analytic topology over $\bb{C}$ to sheaves in the \'etale topology over $\bar{\bb{F}}$, and keeping track of Tate twists arisng from the $\bb{F}$-structure.
The factor of $|U_{w_{J\circ}}^-| = q^{\dim(\bb{L}_J/\bb{B}_J)}$ in the $-$ case arises from a Tate twist of order $2 \dim(\bb{L}_J/\bb{B}_J)$ that accompanies the cohomological shift in case (b) of \cite[\S{3.4}]{bm}.

\subsection{The Bitrace}\label{subsec:bitrace}

As in \S\ref{subsec:trace-intro}, let $O(w)_u$ be the subset of $O(w)$ of pairs taking the form $(hB, uhB)$.
Let $\tau_G : \bb{Q}W \otimes H_B^G \to \bb{Q}$ be defined by
\begin{align}
	\tau_G(z \otimes \sf{1}_w)
	= \frac{1}{|G|}
	\sum_{u \in V}
	|O(w)_u| \chi_u(z).
\end{align}
The framework of \cite{trinh} shows that this is, indeed, a bitrace, meaning $\tau_G(z \otimes \whitearg)$ and $\tau_G(\whitearg \otimes \sf{1}_w)$ are traces for all $z, w \in W$.
In the split case, it recovers the $\X \to q$ specialization of the trace denoted $\tau_G$ in the introduction.

\begin{lem}\label[lem]{lem:tau-g}
For all $J \subseteq S$ and $w \in W$, we have
\begin{align}
\frac{1}{|U_{w_{J\circ}}^-|}\,\tau_G(e_{J, -} \otimes \sf{1}_w) 
	&= \frac{1}{|G|} \sum_{(hB, gB) \in O(w)}
	\mult_!\sf{1}_{E_J^-}(hB, gB),\\
\tau_G(e_{J, +} \otimes \sf{1}_w) 
	&= \frac{1}{|G|} \sum_{(hB, gB) \in O(w)}
	\mult_!\sf{1}_{E_J^+}(hB, gB),
\end{align}
where $E_J^\pm$ and $\mult$ are defined as in \Cref{sec:springer}.
\end{lem}

\begin{proof}
Apply \eqref{eq:bm} to the formula for $\tau_G$.
Then observe that
\begin{align}\label{eq:steinberg-w}
\coprod_{u \in V}
	O(w)_u \times \mathit{Spr}_{J, u}^\pm
	&= 
		\{(hB, u, yP_J) \in E_J^\pm \mid (hB, uhB) \in O(w)\}\\
	&=
		\coprod_{(hB, gB) \in O(w)}
			\mult^{-1}(hB, gB).\hfill\qedhere
\end{align}
\end{proof}

The split case of the following result is the $\X \to q$ specialization of \Cref{thm:trace}.
Since it amounts to a family of identities of Laurent polynomials in $q$, which hold for infinitely many $q$, we can lift it from $q$ to $\X$.

\begin{thm}\label[thm]{thm:trace-geo}
For any $J \subseteq S$, we have
\begin{align}
\tau[N_J^S(1)] 
	&= |T|\,
	\tau_G(e_{J, -} \otimes \whitearg),\\
\tau[N_J^S(\sf{1}_{w_{J\circ}}^2)] 
	&= |B_J|\,
	\tau_G(e_{J, +} \otimes \whitearg)
\end{align}
as traces on $H_W$.
\end{thm} 

\begin{proof}
Combine \Crefrange{lem:tau-norm}{lem:tau-g} with \eqref{eq:main}, noting that $1$ and $\sf{1}_{w_{J\circ}}^2$ are invariant under $\iota$.
Doing so gives
\begin{align}
\frac{1}{|B|}\, \tau[N_J^S(1)] 
	&= \frac{1}{|U_J||U_{w_{J\circ}}^-|}\,
	\tau_G(e_{J, -} \otimes \whitearg)
		= \frac{1}{|U|}\,
		\tau_G(e_{J, -} \otimes \whitearg),\\
\frac{1}{|B|}\, \tau[N_J^S(\sf{1}_{w_{J\circ}}^2)] 
	&= \frac{1}{|U_J|}\,
	\tau_G(e_{J, +} \otimes \whitearg).
\end{align}
Then recall that $B = T \ltimes U = B_J \ltimes U_J$.
\end{proof}

\subsection{The Multiplicity Formula}\label{subsec:exotic}

\emph{Throughout this subsection, we assume that $G$ is split.}
As in \S\ref{subsec:trace-intro}, we write:
\begin{itemize} 
\item 	$\sf{V}_G$ for the representation of $W$ on the $\bb{Q}$-span of the cocharacter lattice of $\bb{T}$.

\item 	$\Irr(W)$ for the set of irreducible characters of $W$.

\item 	$\{-, -\}$ for the truncation of Lusztig's exotic Fourier transform to a $\bb{Q}$-valued pairing on $\Irr(W)$.
		In the notation of \cite{lusztig}, our pairing is the pullback of Lusztig's pairing $\{-, -\}$ along his embedding (4.21.3).

\end{itemize} 
We emphasize that the pairing $\{-, -\}$ remains fairly mysterious.
Notably, its definition in \cite{lusztig} involves some case-by-case constructions.
The most uniform definitions of $\{-, -\}$ involve algebraic geometry.

By \cite{lusztig_81}, $\bb{Q}(\X^{1/2})$ is a splitting field for $H_W$.
Hence, by Tits deformation \cite[Ch.\ 7]{gp}, each character $\chi : W \to \bb{Q}$ defines a trace $\chi_\X : H_W \to \bb{Q}(\X^{1/2})$.
The set of traces $\chi_\X$ with $\chi \in \Irr(W)$ forms a basis for $\bb{Q}(\X^{1/2}) \otimes R(H_W)$ as a vector space.

The character formula in \cite{trinh} translates to an expansion of $\tau_G(z \otimes \whitearg)$ in this basis for any $z \in \bb{Q}W$:
\begin{align}\label{eq:tau-g-to-exotic}
\tau_G(z \otimes \whitearg)
	= \sum_{\chi, \psi \in \Irr(W)}
		\frac{\{\chi, \psi\}  \psi(z)}{\det(\X - z \mid \sf{V}_G)}\, \chi_\X.
\end{align}
Combining this with \Cref{thm:trace} gives \Cref{cor:exotic}.

\subsection{Recovering Lascoux--Wan--Wang}\label{subsec:lascoux}

In this subsection, we take $\bb{G} = \mathbf{GL}_n$, and $F$ to be the standard Frobenius that raises each matrix coordinate to its $q$th power.
Then $G = \GL_n(\bb{F})$ and $W = \bb{W} = S_n$.
For each integer partition $\lambda \vdash n$, let $\chi^\lambda \in \Irr(S_n)$ be the corresponding irreducible character.
The trace $\chi_\X^\lambda$ turns out to be $\bb{Q}(\X)$-valued, not just $\bb{Q}(\X^{1/2})$-valued, so the map $\FC_\X$ in \S\ref{subsec:trace-intro} is well-defined.

As in \emph{loc.\ cit.}, we take $S = \{s_1, \ldots, s_{n - 1}\}$, where $s_i \in S_n$ is the transposition swapping $i$ and $i + 1$.
We will use the bijection between integer compositions of $n$ and subsets of $S$ that matches $\nu = (\nu_1, \nu_2, \ldots) \vdash n$ with
\begin{align}
J = S \setminus \{s_{\nu_1}, s_{\nu_1 + \nu_2}, \ldots\}
\end{align}
For this $J$, we find that $W_J \subseteq W$ is the \dfemph{Young subgroup} $S_\nu \simeq S_{\nu_1} \times S_{\nu_2} \times \ldots$

For $G = \GL_n(\bb{F})$, the pairing $\{-, -\}$ in \S\ref{subsec:exotic} is given by $\{\chi, \chi\} = 1$ and $\{\chi, \psi\} = 0$ whenever $\chi \neq \psi$.
So to prove that \Cref{cor:exotic} recovers Wan--Wang's formulas \eqref{eq:lascoux}, it remains to prove:

\begin{prop}\label[prop]{prop:tau-to-e-h}
If the subset $J$ corresponds to the integer composition $\nu$, then
\begin{align} 
\frac{\chi^\lambda(e_{J, -})}{\det(\X - e_{J, -} \mid \sf{V}_G)}
	&= \left\langle s_\lambda[X], e_\nu\left[\frac{X}{\X - 1}\right]\right\rangle,\\[0.5ex]
\frac{\chi^\lambda(e_{J, +})}{\det(\X - e_{J, +} \mid \sf{V}_G)}
	&= \left\langle s_\lambda[X], h_\nu\left[\frac{X}{\X - 1}\right]\right\rangle
\end{align}
for any $\lambda \vdash n$, where $\langle -, -\rangle$ is the Hall pairing on $\Lambda_n$ in which the Schur functions $s_\lambda[X]$ are orthonormal.
\end{prop}

As preparation, let $R(S_n)$ be the vector space of $\bb{Q}(\X)$-valued traces on $\bb{Q}S_n$.
Let $\FC : R(S_n) \xrightarrow{\sim} \Lambda_n$ be the \dfemph{(undeformed) Frobenius characteristic} isomorphism that sends $\chi^\lambda$ to $s_\lambda[X]$, and the multiplicity pairing on $R(S_n)$ to the Hall pairing.

\begin{proof}
Recall that $\FC$ sends $\chi^\lambda/{\det(\X - \whitearg \mid \sf{V}_G)}$ to the plethystically transformed Schur $s_\lambda[\frac{X}{\X - 1}]$.
At the same time, since $W_J = S_\lambda$, it sends the induced character of $W = S_n$ arising from the trivial, \emph{resp.}\ sign, character of $W_J$ to the symmetric function $h_\nu[X]$, \emph{resp.}\ $e_\nu[X]$.
So by Frobenius reciprocity,
\begin{align}
\frac{\chi^\lambda(e_{J, +})}{\det(\X - e_{J, +} \mid \sf{V}_G)}
	&= \left\langle 
		s_\lambda\left[\frac{X}{\X - 1}\right], h_\nu[X]
		\right\rangle
	= \left\langle 
	s_\lambda[X], h_\nu\left[\frac{X}{\X - 1}\right]
	\right\rangle,
\end{align}
and similarly with $e_{J, -}$, $e_\nu$ in place of $e_{J, +}$, $h_\nu$.
\end{proof}

\section{Braid Varieties and Deodhar Decompositions}\label{sec:cell}

\subsection{}

\emph{For the rest of the paper, we assume that $G$ is split}.
In this section, we prove \Cref{thm:cell}, relating partial braid Steinberg varieties to the cell decompositions of open braid Richardson varieties.
In fact, we prove a refinement that respects individual cells.

We will freely use the terminology from Coxeter combinatorics that we reviewed in \S\ref{subsec:cell-intro}.
Throughout, we fix a word $\vec{s} = (s^{(1)}, \ldots, s^{(\ell)})$ in $S$.

\subsection{Richardson Varieties}

Recall that for any $v \in W$, we defined the \dfemph{$v$-twisted open Richardson variety} of $\vec{s}$ on points by
\begin{align}
\bb{R}^{(v)}(\vec{s})
	= \{\vec{g}\bb{B} \in \bb{O}(\vec{s}) \mid \text{$g_0\bb{B} =  vw_\circ\bb{B}$ and $\bb{B} \xrightarrow{vw_\circ} g_\ell\bb{B}$}\}.
\end{align}
Below, we give further detail about the cell decomposition mentioned in \S\ref{subsec:cell-intro}.
For any $v$-distinguished subword $\vec{\omega}$ of $\vec{s}$, let $\bb{R}^{(v)}(\vec{s})_{\vec{\omega}} \subseteq \bb{R}^{(v)}(\vec{s})$ be the subvariety
\begin{align}
	\bb{R}^{(v)}(\vec{s})_{\vec{\omega}}
	&= \{\vec{g}\bb{B} \in \mathbf{R}^{(v)}(\vec{s}) \mid \bb{B} \xrightarrow{v\omega_{(i)}w_\circ} g_i\bb{B}\}.
\end{align}
As before, let $\cal{D}^{(v)}(\vec{s})$ be the set of $v$-distinguished subwords $\vec{\omega}$ of $\vec{s}$ such that $\omega_{(\ell)} = e$.
For any $\vec{\omega} \in \cal{D}^{(v)}(\vec{s})$, let
\begin{align}
\mathbf{d}_{\vec{\omega}}
	&= \{i \mid v\omega_{(i)} < v\omega_{(i - 1)}\},\\
\mathbf{e}_{\vec{\omega}} 
	&= \{i \mid \omega^{(i)} = e\},
\end{align} 
The main results of \cite{deodhar} show that for any word $\vec{s}$ in $S$:
\begin{enumerate} 
	\item 	$\bb{R}^{(v)}(\vec{s})_{\vec{\omega}}$ is nonempty if and only if $\omega \in \mathcal{D}^{(v)}(\vec{s})$.
	In this case,
	\begin{align}\label{eq:cell}
		\bb{R}^{(v)}(\vec{s})_{\vec{\omega}}
		&\simeq \left\{\vec{t} \in \mathbf{A}^\ell\, \middle|
		\begin{array}{ll}
			t_i \neq 0
			&\text{for $i \in \mathbf{e}_{\vec{\omega}}$},\\
			t_i= 0
			&\text{for $i \notin \mathbf{d}_{\vec{\omega}} \cup \mathbf{e}_{\vec{\omega}}$}
		\end{array}\!\right\}
	\end{align}
	from which $R^{(v)}(\vec{s})_{\vec{\omega}} \vcentcolon= \bb{R}^{(v)}(\vec{s})_{\vec{\omega}}^F$ satisfies
	\begin{align}\label{eq:cell-r-s-omega}
	|R^{(v)}(\vec{s})_{\vec{\omega}}| 
		= q^{|\mathbf{d}_{\vec{\omega}}|}
		(q - 1)^{|\mathbf{e}_{\vec{\omega}}|}.
	\end{align} 

	\item 	The subvarieties $\bb{R}^{(v)}(\vec{s})_{\vec{\omega}}$ are pairwise disjoint and partition $\mathbf{R}^{(v)}(\vec{s})$ as we run over $\vec{\omega} \in \mathcal{D}^{(v)}(\vec{s})$.
	
\end{enumerate}
In light of \eqref{eq:cell}, the varieties $\bb{R}^{(v)}(\vec{s})_{\vec{\omega}}$ are called \dfemph{Deodhar cells}.

\subsection{Change of Structure Group}\label{subsec:change-of-group}

To compare them to the geometry in previous sections, we need a more symmetrical version of the open Richardson varieties.
Let $\bb{X}^{(v)}$, $\bb{X}_\bb{B}^{(v)}$, $\bb{R}^{(v)}$ be the varieties defined on points by
\begin{align}
\bb{X}^{(v)}
	&= \{(h\bb{B}, x\bb{B}, g\bb{B}) \in (\bb{G}/\bb{B})^3 \mid h\bb{B} \xleftarrow{vw_\circ} x\bb{B} \xrightarrow{vw_\circ} g\bb{B}\}\\
	&\simeq \bb{O}((vw_\circ)^{-1}, vw_\circ),\\
\bb{X}_\bb{B}^{(v)}
	&= \{(h\bb{B}, g\bb{B}) \in (\bb{G}/\bb{B})^2 \mid h\bb{B} \xleftarrow{vw_\circ} \bb{B} \xrightarrow{vw_\circ} g\bb{B}\},\\
\bb{R}^{(v)}
	&= \{vw_\circ\bb{B}\} \times \bb{B}vw_\circ \bb{B}/\bb{B}.
\end{align}
By construction, $\bb{R}^{(v)}(\vec{s})$
is the preimage of $\bb{R}^{(v)}$ along $\bb{O}(\vec{s}) \xrightarrow{\pr_{0, \ell}} (\bb{G}/\bb{B})^2$.
We will relate the varieties above to one another, thereby relating $\bb{R}^{(v)}(\vec{s})$ and its Deodhar cells to analogous varieties built from $\bb{X}^{(v)}$, $\bb{X}_\bb{B}^{(v)}$.

Observe that $\bb{X}^{(v)}$ is stable under the $\bb{G}$-action on $(\bb{G}/\bb{B})^3$.
The action of $\bb{G}$ on $\bb{X}^{(v)}$ restricts to an action of $\bb{B}$ on $\bb{X}_\bb{B}^{(v)}$, which in turn restricts to an action of
\begin{align} 
\bb{B}_v^- \vcentcolon= \bb{B} \cap v\bb{B}_-v^{-1} = \bb{B} \cap (vw_\circ)\bb{B}(vw_\circ)^{-1}
\end{align} 
on $\bb{R}^{(v)}$.
By \Cref{lem:bruhat}(2), $\bb{B} = \bb{B}_v^-\bb{U}_v = \bb{U}_v \bb{B}_v^-$ and $\bb{B}_v^- \cap \bb{U}_v = \{1\}$.

\begin{lem}\label[lem]{lem:action}
For any $v \in W$, let $\bb{B}$ act on $\bb{G} \times \bb{X}_\bb{B}^{(v)}$ from the left by 
\begin{align} 
b \cdot (x, h\bb{B}, g\bb{B}) = (xb^{-1}, bh\bb{B}, bg\bb{B}).
\end{align}
Then:
\begin{enumerate}
	\item 	The map $(\bb{G} \times \bb{X}_\bb{B}^{(v)})/\bb{B} \to \bb{X}^{(v)}$ that sends $[x, h\bb{B}, g\bb{B}] \mapsto (xh\bb{B}, x\bb{B}, xg\bb{B})$ is an isomorphism.
	
	\item 	The quotient $\bb{X}_\bb{B}^{(v)}/\bb{U}_v$ forms an algebraic variety.
	The composition of maps
	\begin{align} 
		\bb{R}^{(v)} \to \bb{X}_\bb{B}^{(v)} \to \bb{X}_\bb{B}^{(v)}/\bb{U}_v
	\end{align}
	is an isomorphism.
	
\end{enumerate}
\end{lem}

\begin{proof}
	(1):
	$\bb{X}_\bb{B}^{(v)}$ is the closed subvariety of $\bb{X}^{(v)}$ cut out by the condition $x\bb{B} = \bb{B}$.
	The $\bb{G}$-action on $\bb{X}^{(v)}$ is transitive on the coordinate $x\bb{B}$, and the stabilizer of the point $\bb{B}$ is itself.
	
	(2):
	$\bb{R}^{(v)}$ is the closed subvariety of $\bb{X}_\bb{B}^{(v)}$ cut out by the condition $h\bb{B} = vw_\circ\bb{B}$.
	By \Cref{lem:bruhat}(3), the $\bb{B}$-action on $\bb{X}_\bb{B}^{(v)}$ restricts to an action of $\bb{U}_{vw_\circ}^- = \bb{U}_v$ that is simply transitive on the coordinate $h\bb{B}$.
\end{proof}

\begin{cor}\label[cor]{cor:action}
The maps $(G \times X_B^{(v)})/B \to X^{(v)}$ and $R^{(v)} \to X_B^{(v)}/U_v$ on $F$-fixed points induced by the isomorphisms above are bijections.
\end{cor}

\begin{proof}
Immediate from Lang's theorem, since $\bb{B}$, \emph{resp.}\ $\bb{U}_v$, is connected and acts freely on $\bb{G} \times \bb{X}_\bb{B}^{(v)}$, \emph{resp.}\ $\bb{X}_\bb{B}^{(v)}$.
\end{proof}

Let $\bb{X}^{(v)}(\vec{s}) = \bb{O}(\vec{s}) \times_{(\bb{G}/\bb{B})^2} \bb{X}^{(v)}$ and $\bb{X}_\bb{B}^{(v)}(\vec{s}) = \bb{O}(\vec{s}) \times_{(\bb{G}/\bb{B})^2} \bb{X}_\bb{B}^{(v)}$, where the fiber products are formed with respect to the maps $\pr_{0, \ell}$ on the left factors and the coordinate pairs $(h\bb{B}, g\bb{B})$ on the right factors.
On points,
\begin{align}
\bb{X}^{(v)}(\vec{s})
	&= \{(\vec{g}\bb{B}, x\bb{B}) \in \bb{O}(\vec{s}) \times \bb{G}/\bb{B} \mid g_0\bb{B} \xleftarrow{vw_\circ} x\bb{B} \xrightarrow{vw_\circ} g_\ell\bb{B}\},\\
\bb{X}_\bb{B}^{(v)}(\vec{s})
	&= \{\vec{g}\bb{B} \in \bb{O}(\vec{s}) \mid g_0\bb{B} \xleftarrow{vw_\circ} \bb{B} \xrightarrow{vw_\circ} g_\ell\bb{B}\}.
\end{align}
These varieties can respectively be partitioned into subvarieties
\begin{align}
	\bb{X}^{(v)}(\vec{s})_{\vec{\omega}}
	&= \{(\vec{g}\bb{B}, x\bb{B}) \in \bb{O}(\vec{s}) \times \bb{G}/\bb{B} \mid x\bb{B} \xrightarrow{v\omega_{(i)}w_\circ} g_i\bb{B}\},\\
	\bb{X}_\bb{B}^{(v)}(\vec{s})_{\vec{\omega}}
	&= \{\vec{g}\bb{B} \in \bb{X}^{(v)}(\vec{s}) \mid \bb{B} \xrightarrow{v\omega_{(i)}w_\circ} g_i\bb{B}\}
\end{align}
as $\vec{\omega}$ runs over $\mathcal{D}^{(v)}(\vec{s})$.
Note that $\bb{X}^{(v)}(\vec{s})_{\vec{\omega}}$ is stable under the $\bb{G}$-action on $\bb{X}^{(v)}(\vec{s})$, as are $\bb{X}_\bb{B}^{(v)}(\vec{s})_{\vec{\omega}}$, \emph{resp.}\ $\bb{R}^{(v)}(\vec{s})_{\vec{\omega}}$, under $\bb{B}$, \emph{resp.}\ $\bb{B}_{vw_\circ}$.
Pulling back \Cref{lem:action} along $\pr_{0, \ell} : \bb{O}(\vec{s})_{\vec{\omega}} \to (\bb{G}/\bb{B})^2$, we see:

\begin{cor}\label[cor]{cor:action-s}
For any $\vec{s}$ and $\vec{\omega} \in \mathcal{D}^{(v)}(\vec{s})$, the analogues of \Cref{lem:action} and \Cref{cor:action} hold with $\diamondsuit(\vec{s})_{\vec{\omega}}$ replacing $\diamondsuit$ for each $\diamondsuit \in \{ \bb{X}^{(v)}, \bb{X}_\bb{B}^{(v)}, \bb{R}^{(v)}\}$.
Thus,
\begin{align}
|X^{(v)}(\vec{s})_{\vec{\omega}}| 
	&= \frac{|G||X_B^{(v)}(\vec{s})_{\vec{\omega}}|}{|B|},\\
|X_B^{(v)}(\vec{s})_{\vec{\omega}}| 
	&= |U_v||R^{(v)}(\vec{s})_{\vec{\omega}}|.
\end{align}
\end{cor}

\subsection{Steinberg Varieties}

Fix $J \subseteq S$.
As in \S\ref{subsec:cell-intro}, we define the \dfemph{partial Steinberg varieties} of $\vec{s}$ of type $J$ on points by
\begin{align}
\bb{Z}_J^\pm(\vec{s})
	= \{(\vec{g}\bb{B}, u, y\bb{P}_J) \in \bb{O}(\vec{s}) \times \bb{Spr}_J^\pm \mid 
	ug_0\bb{B} = g_\ell\bb{B}
	\}.
\end{align}
We let $\bb{G}$ act on $\bb{Z}_J^\pm(\vec{s})$ via its actions on $\bb{Spr}_J^\pm$ and $\bb{O}(\vec{s})$.
The coordinate triple $(g_\ell\bb{B}, u, y\bb{P}_J)$ defines an equivariant map $\bb{Z}_J^\pm(\vec{s}) \to \bb{E}_J^\pm$.
Pulling back the partition of $\bb{E}_J^\pm$ by subvarieties $\bb{E}_{J, v}^\pm$ in \Cref{sec:springer}, we get a partition of $\bb{Z}_J^\pm(\vec{s})$ into subvarieties
\begin{align}
\bb{Z}_{J, v}^\pm(\vec{s})
	= \{(\vec{g}\bb{B}, u, y\bb{P}_J) \in \bb{Z}_J^\pm(\vec{s}) \mid 
	\bb{P}_Jy^{-1}g_\ell\bb{B} = \bb{P}_Jv\bb{B}
	\}
\end{align}
as $W_Jv$ runs over $W_J\backslash W$.
Note that the points of $\bb{Z}_{J, v}^\pm(\vec{s})$ also satisfy the condition $\bb{P}_J y^{-1}g_0\bb{B} = \bb{P}_Jv\bb{B}$.

\begin{prop}\label[prop]{prop:z-x-s}
If $v \in W^{J, -}$, then:
\begin{enumerate} 
\item 	$|Z_{J, v}^-(\vec{s})| = |U_{w_{J\circ}v}| |X^{(vw_\circ)}(\vec{s})|$.\\[-3ex]

\item 	$|Z_{J, v}^+(\vec{s})| = |U_{w_{J\circ}v}| |X^{(w_{J\circ}vw_\circ)}(\vec{s})|$.

\end{enumerate}
\end{prop}

\begin{proof}
For any $v \in W$, we have
\begin{align} 
	\label{eq:z-sum}
	|Z_{J, v}^\pm(\vec{s})|
	&= \sum_{\vec{g}B \in O(\vec{s})} 
	\mult_!\sf{1}_{E_{J, v}^\pm}(g_0B, g_\ell B),\\
	\label{eq:x-sum}
|X^{(vw_\circ)}(\vec{s})|
	&= \sum_{\vec{g}B \in O(\vec{s})} 
	(\sf{1}_{v^{-1}} \ast \sf{1}_{v\vphantom{^{-1}}})(g_0B, g_\ell B).
\end{align} 
(The second identity used the involutivity of $w_\circ$.)
Now apply \Cref{thm:main-v}.
\end{proof}

Since multiplication by $w_\circ$ or $w_{J\circ}$ swaps $W^{J, -}$ with $W^{J, +}$, the following result implies \Cref{thm:cell}.

\begin{cor}\label[cor]{cor:z-r-s}
If $v \in W^{J, -}$, then
\begin{align}
\frac{|Z_{J, v}^-(\vec{s})|}{|G|}
	&= \frac{1}{q^{\ell_J} (q - 1)^{\ur{rk}(G)}}
		\sum_{\vec{\omega} \in \mathcal{D}^{(v)}(\vec{s})}
		q^{|\sf{d}_{\vec{\omega}}|} 
		(q - 1)^{|\sf{e}_{\vec{\omega}}|},\\
\frac{|Z_{J, v}^+(\vec{s})|}{|G|}
	&= \frac{1}{(q - 1)^{\ur{rk}(G)}}
	\sum_{\vec{\omega} \in \mathcal{D}^{(v)}(\vec{s})}
	q^{|\sf{d}_{\vec{\omega}}|} 
	(q - 1)^{|\sf{e}_{\vec{\omega}}|}.
\end{align}
\end{cor}

\begin{proof}
We only do the $-$ case, as the $+$ case is similar.
Observe that
\begin{align}
\frac{|U_{w_{J\circ}v}||X^{(vw_\circ)}(\vec{s})|}{|G|}
	= 
	\frac{|U_{w_{J\circ}v}||U_v||R^{(vw_\circ)}(\vec{s})|}{|B|}
	= 
	\frac{|U_J||R^{(vw_\circ)}(\vec{s})|}{|B|}
	= 
	\frac{|R^{(vw_\circ)}(\vec{s})|}{|B_J|}
\end{align}
by \Cref{cor:action-s} and \Cref{lem:j}.
Then apply \Cref{prop:z-x-s} on the left and \eqref{eq:cell-r-s-omega} on the right.
\end{proof}

\begin{rem}
It is not always the case that $|Z_{J, v}^-(\vec{s})/G| = |Z_{J, v}^-(\vec{s})|/|G|$.
Indeed, the $\bb{G}$-action on $\bb{Z}_{J, v}^\pm(\vec{s})$ need not be free, so we cannot apply Lang's theorem.
\end{rem}

\subsection{Traces as Point Counts} \label{subsec:traces-as-counts}

We collect point-counting formulas for specific traces.
Let $\sf{1}_{\vec{s}} = \sf{1}_{s^{(1)}} \ast \cdots \ast \sf{1}_{s^{(\ell)}}$.
Summing \eqref{eq:z-sum} over $W_Jv$ and applying \Cref{lem:tau-g} yields
\begin{align}\label{eq:z-trace}
\tau_G(e_{J, \pm} \otimes \sf{1}_{\vec{s}})
	&= \frac{|Z_J^\pm(\vec{s})|}{|G|}.
\end{align}
Similarly, for any $v \in W$, \eqref{eq:x-sum} yields
\begin{align}\label{eq:x-trace}
\frac{1}{|B|}\, \tau(\sf{1}_{\vec{s}} \ast \sf{1}_{v^{-1}} \ast \sf{1}_{v\vphantom{^{-1}}})
	&= \frac{|X^{(vw_\circ)}(\vec{s})|}{|G|}.
\end{align}
For the purpose of proving \Cref{thm:cell}, we do not actually need these results.
But in later sections, it will be useful to have a $\X$-version of the formula
\begin{align}\label{eq:gltw}
q^{-\ell(v)}
	\tau(\sf{1}_{\vec{s}} \ast \sf{1}_{v^{-1}} \ast \sf{1}_{v\vphantom{^{-1}}})
	&= |R^{vw_\circ}(\vec{s})|
\end{align}
that follows from combining \eqref{eq:x-trace}, \Cref{cor:action-s}, and \Cref{lem:rosenlicht}.
This formula is itself an easier version of Corollary 5.3 in \cite{gltw}.

Namely:
Let $T_{\vec{s}} = T_{s^{(1)}} \cdots T_{s^{(\ell)}}$.
Combining \eqref{eq:cell-r-s} and \eqref{eq:gltw} gives an identity of Laurent polynomials in $\sf{1}_{\vec{s}}$ and $q$ that holds for infinitely many $q$, hence lifts to
\begin{align}\label{eq:gltw-generic}
\X^{-\ell(v)}
\tau(T_{\vec{s}} T_{v^{-1}} T_{v\vphantom{^{-1}}})
	&= \sum_{\vec{\omega} \in \mathcal{D}^{(vw_\circ)}(\vec{s})}
	\X^{|\sf{d}_{\vec{\omega}}|} 
	(\X - 1)^{|\sf{e}_{\vec{\omega}}|},
\end{align}
an identity in $T_{\vec{s}}$ and $\X$.

\subsection{Decomposing Steinberg Varieties}

We can significantly refine case (1) of \Cref{prop:z-x-s}.
For any $\vec{\omega} \in \cal{D}^{(vw_\circ)}(\vec{s})$, let $\bb{Z}_{J, v}^\pm(\vec{s})_{\vec{\omega}}$ be the $\bb{G}$-stable subvariety of $\bb{Z}_{J, v}^\pm(\vec{s})$ defined by
\begin{align}
\bb{Z}_{J, v}^\pm(\vec{s})_{\vec{\omega}}
= \{(\vec{g}\bb{B}, u, y\bb{P}_J) \in \bb{Z}_{J, v}^\pm(\vec{s}) \mid 
\bb{P}_Jy^{-1}g_i\bb{B} = \bb{P}_J vw_\circ\omega_{(i)}w_\circ\bb{B}
\}.
\end{align}
This subvariety only depends on $W_Jvw_\circ$, even though $\vec{\omega}$ depends on $vw_\circ$ itself.
For any $v \in W$, let $\check{\bb{Z}}_{J, v}^-(\vec{s})_{\vec{\omega}}$ be the pullback of $\bb{Z}_{J, v}^-(\vec{s})_{\vec{\omega}}$ along the forgetful map $\check{\bb{E}}_{J, v}^- \to \bb{E}_{J, v}^-$ from \S\ref{subsec:p-to-b}.
By pulling back \Cref{lem:check} and \Cref{prop:minus-case} along $\pr_{0, \ell} : \bb{O}(\vec{s})_{\vec{\omega}} \to (\bb{G}/\bb{B})^2$, we obtain:
 
\begin{prop}\label[prop]{prop:cartesian}
If $v \in W^{J, -}$ and $\vec{\omega} \in \cal{D}^{(vw_\circ)}(\vec{s})$, then the maps $\check{\bb{E}}_{J, v}^- \xrightarrow{\sim} \bb{E}_{J, v}^-$ and $\check{\mult} : \check{\bb{E}}_{J, v}^- \to \bb{O}(v^{-1}, v) = \bb{X}^{(vw_\circ)}$ of \S\ref{subsec:p-to-b} fit into a cartesian diagram:
\begin{equation}\label{eq:z-e-x}
\begin{tikzcd}
\bb{Z}_{J, v}^-(\vec{s})_{\vec{\omega}} \ar{r}
	&\bb{E}_{J, v}^-\\
\check{\bb{Z}}_{J, v}^-(\vec{s})_{\vec{\omega}} \ar[u, "\sim" {anchor=south, rotate=90}] \ar{r} \ar{d}
	&\check{\bb{E}}_{J, v}^- \ar[u, "\sim" {anchor=south, rotate=270}] \ar{d}{\check{\mult}}\\
\bb{X}^{(vw_\circ)}(\vec{s})_{\vec{\omega}} \ar{r}
	&\bb{X}^{(vw_\circ)}
\end{tikzcd}
\end{equation}
Hence, $\check{\bb{Z}}_{J, v}^-(\vec{s})_{\vec{\omega}} \to \bb{X}^{(vw_\circ)}(\vec{s})_{\vec{\omega}}$ forms a smooth fiber bundle that restricts to a $\bb{U}_{w_{J\circ}v}$-torsor over the subvariety $(h\bb{B}, x\bb{B}) = (v\bb{B}, \bb{B})$.
\end{prop}

\begin{cor}
If $v \in W^{J, -}$, then the $\bb{Z}_{J, v}^\pm(\vec{s})_{\vec{\omega}}$ are pairwise disjoint and partition $\bb{Z}_{J, v}^\pm(\vec{s})$ as $\vec{\omega}$ runs over $\cal{D}^{(vw_\circ)}(\vec{s})$.
\end{cor} 

\begin{proof} 
\Cref{prop:cartesian} shows that if $v \in W^{J, -}$, then $\bb{Z}_{J, v}^-(\vec{s})_{\vec{\omega}}$ arises from $\bb{X}^{(vw_\circ)}(\vec{s})_{\vec{\omega}}$ by pullback.
This establishes the statement for the $-$ case.
But the condition defining $\bb{Z}_{J, v}^\pm(\vec{s})_{\vec{\omega}} \subseteq \bb{Z}_{J, v}^\pm(\vec{s})$ does not involve the coordinate $u$ by which $\bb{Z}_{J, v}^-(\vec{s})$ and $\bb{Z}_{J, v}^+(\vec{s})$ differ.
So we also get the statement for the $+$ case.
\end{proof}

\begin{cor}\label[cor]{cor:cohomology}
If $v \in W^{J, -}$ and $\vec{\omega} \in \cal{D}^{vw_\circ}(\vec{s})$, then
\begin{align} 
|Z_{J, v}^-(\vec{s})_{\vec{\omega}}| 
	&= |U_{w_{J\circ}v}| |X^{(vw_\circ)}(\vec{s})|,\\
	&= |G|\, q^{|\sf{d}_{\vec{\omega}}| - \ell_J} 
		(q - 1)^{|\sf{e}_{\vec{\omega}}| - \ur{rk}(G)},
\end{align} 
refining the $-$ cases of \Cref{prop:z-x-s} and \Cref{cor:z-r-s}.

Moreover, the $\bb{G}$-equivariant \'etale cohomology of $\bb{Z}_{J, v}^-(\vec{s})_{\vec{\omega}}$ with $\QL$-coefficients is isomorphic to the $\bb{T}$-equivariant \'etale cohomology of $\bb{R}^{(vw_\circ)}(\vec{s})_{\vec{\omega}}$.
The analogous statement for compactly-supported cohomology holds up to a shift of degree $\ell_J$.

\end{cor} 

\begin{proof}
The	first claim follows from \Cref{prop:cartesian} by taking $F$-fixed points.
As for the second, let $\ur{H}_c^\ast$ denote compactly-supported \'etale cohomology.
Then
\begin{align}
&\ur{H}_{c, \bb{G}}^\ast(\bb{Z}_{J, v}^-(\vec{s})_{\vec{\omega}})\\
	&\simeq \ur{H}_{c, \bb{G}}^\ast(\bb{X}^{(vw_\circ)}(\vec{s})_{\vec{\omega}})[\dim \bb{U}_{w_{J\circ}v}]
		&&\text{by \Cref{prop:cartesian}}\\
	&\simeq \ur{H}_{c, \bb{B}}^\ast(\bb{R}^{(vw_\circ)}(\vec{s})_{\vec{\omega}})[\dim \bb{U}_{w_{J\circ}v} + \dim \bb{U}_{vw_\circ}]
		&&\text{by \Cref{cor:action-s}}\\
	&\simeq \ur{H}_{c, \bb{T}}^\ast(\bb{R}^{(vw_\circ)}(\vec{s})_{\vec{\omega}})[\dim \bb{U}_{w_{J\circ}v} + \dim \bb{U}_{vw_\circ} - \dim \bb{U}]
	&&\text{since $\bb{B} = \bb{T} \ltimes \bb{U}$}.
\end{align}
Finally, $\dim \bb{U}_{w_{J\circ}v} + \dim \bb{U}_{vw_\circ} - \dim \bb{U} = -\ell_J$ by \Cref{lem:bruhat}.
The statements for ordinary cohomology are the same, except there are no shifts.
\end{proof}

\begin{rem}\label[rem]{rem:stz}
When $J = S$, we must have $v = e$ in \eqref{eq:z-e-x}.
Here the vertical arrows become trivial, giving isomorphisms $\bb{E}_S^- \simeq \check{\bb{E}}_S^- \simeq \bb{G}/\bb{B}$ and $\bb{Z}_S^-(\vec{s}) \simeq \bb{X}^{(w_\circ)}(\vec{s})$.

When $G = \PGL_n(\bb{F})$, so that $W = S_n$, and $\beta$ is the positive braid on $n$ strands defined by $\vec{s}$, the stack denoted $\cal{M}(\beta^\circ)$ in \cite{stz} is precisely $[\bb{X}^{(w_\circ)}(\vec{s})/\bb{G}]$.
Their Proposition 6.31 gives a decomposition of another stack $\cal{M}(\beta^\succ)$ into substacks indexed by rulings of a Legendrian link $\beta^\succ$.
At the same time,
\begin{align} 
\cal{M}(\beta^\succ) \simeq \cal{M}((\Delta\beta\Delta)^\circ) \simeq \cal{M}((\beta \Delta^2)^\circ),
\end{align} 
where $\Delta$ is the \dfemph{half-twist}: the minimal positive braid that lifts $w_\circ \in S_n$.
Note that $\cal{M}((\Delta\beta\Delta)^\circ)$ is also isomorphic to $[\bb{X}^{(e)}(\vec{s})/\bb{G}]$.

In this way, our stacks $[\bb{Z}_{J, v}^-(\vec{s})/\bb{G}]$ generalize the stacks $\cal{M}(\beta^\circ)$ and $\cal{M}(\beta^\succ)$ in \cite{stz}.
Our decomposition of $\bb{Z}_{J, v}^-(\vec{s})$ into subvarieties $\bb{Z}_{J, v}^-(\vec{s})_{\vec{\omega}}$ generalizes their ruling decomposition of $\cal{M}(\beta^\succ)$:
Indeed, \Cref{lem:action} and \Cref{prop:cartesian} show that the former corresponds to the Deodhar decomposition of $\bb{R}^{(v)}(\vec{s})$ under change of structure group from $\bb{G}$ to $\bb{B}/\bb{U}_v$, while \cite{achllw} shows that the latter corresponds to the Deodhar decomposition under change of structure group from $\bb{G}$ to $\bb{T}$.
\end{rem} 

\begin{rem}\label[rem]{rem:mellit}
In \cite{mellit}, Mellit proves the curious Lefschetz property for tame $\bb{GL}_n$ character varieties---or more precisely, their complex analogues---by first decomposing (vector bundles over) them into braid varieties, and the latter into Deodhar cells.\footnote{Note that Mellit's paper uses the term \emph{stratification} for the Deodhar decomposition.
Nonetheless, Dudas showed in \cite{dudas} that Deodhar decompositions in type $B$ need not be stratifications in the technical sense:
The Zariski closure of a cell need not be a union of cells.}
The braid varieties are denoted $Y_\beta(t)$, where $\beta$ is a positive braid on $n$ strands represented by an explicit word in $W$, and $t$ is a sufficiently generic element of the maximal torus $\bb{T}$.

Due to the genericity condition, $Y_\beta(t)$ is qualitatively different from the varieties that we discussed earlier in this section, even up to change of structure group.
It is essentially the fiber at $t$ of a certain formal monodromy map from the open Richardson variety of $\beta$ into $\bb{T}$.
(The use of the Springer resolution in \cite[\S{8}]{mellit} is unrelated to ours.)

Nonetheless, at the level of the tame character variety, Mellit's work does involve structure related to relative norms.
To explain, suppose that the character variety is built from a Riemann surface of genus zero, with $k + 1$ punctures, where the monodromy conditions are specified by semisimple conjugacy classes in $\bb{G}$:
say, $[C_i]$ for $1 \leq i \leq k + 1$, with $C_i$ in $\bb{T}$ for all $i$.
For each $i$, let $J_i \subseteq S$ be the subset of reflections that fix $C_i$.
Theorem 7.3.1 of \cite{mellit} states that when the $[C_i]$ satisfy a certain genericity assumption, there is a vector bundle with fiber $\bb{U}$ over the character variety that is in turn a disjoint union of braid varieties indexed by sequences $\vec{v} = (v_1, \ldots, v_k)$, where $v_i$ runs over $W^{J_i, -}$ for each $i$.
The positive braid $\beta$ corresponding to $\vec{v}$ is represented by the following word in $W$:
\begin{align}
(v_1^{-1}, v_1, \ldots, v_k^{-1}, v_k).
\end{align}
Thus, the decomposition of (the vector bundle over) the character variety into braid varieties is curiously similar to our decomposition of $\bb{Z}_J^\pm(\vec{s})$ into its subvarieties  $\bb{Z}_{J, v}^\pm(\vec{s})$, especially when $k = 1$ in the setup above.
\end{rem}

\subsection{Framed Steinberg Varieties}

In \S\ref{subsec:change-of-group}, the passage from $\bb{X}^{(v)}$ to $\bb{X}_\bb{B}^{(v)}$ to $\bb{R}^{(v)}$ encoded a passage from $\bb{G}$-symmetry to $\bb{B}$-symmetry to $\bb{B}_v^-$-symmetry.
Instead of the $\bb{G}$-varieties $\bb{Z}_J^\pm(\vec{s})$ and their strata, we could have used $\bb{P}_J$-varieties
\begin{align}
\bb{Z}_{J, \square}^-(\vec{s})
	&= \{(\vec{g}\bb{B}, u) \in \bb{O}(\vec{s}) \times \bb{U}_J \mid 
	ug_0\bb{B} = g_\ell\bb{B}
	\},\\
\bb{Z}_{J, \square}^+(\vec{s})
	&= \{(\vec{g}\bb{B}, u) \in \bb{O}(\vec{s}) \times \bb{V}_J \mid 
	ug_0\bb{B} = g_\ell\bb{B}
	\}
\end{align}
and strata cut out by conditions of the form $\bb{P}_Jg_0\bb{B} = \bb{P}_Jg_\ell\bb{B} = \bb{P}_Jv\bb{B}$.

Analogues of \Cref{prop:cartesian} and its corollaries hold for the $\square$ versions.
In fact, the $\bb{G}$-equivariant cohomology of $\bb{Z}_{J, v}^\pm(\vec{s})_{\vec{\omega}}$ matches the $\bb{P}_J$-equivariant cohomology of its $\square$-analogue, by construction.

\section{Parking Numbers}\label{sec:parking}

\subsection{}

\emph{In this subsection and the next, $(W, S)$ denotes an arbitrary irreducible, finite Coxeter system with Coxeter number $h$}.
We write $\sf{V}$ to denote the irreducible reflection representation of $W$, and $\chi_{\sf{V}}$ to denote its character.

For any integer $p$, let $\sf{V}_p$ denote the \dfemph{Galois conjugate} of $\sf{V}$ that has the same underlying vector space but character given by $\chi_{\sf{V}_p}(w) = \chi_\sf{V}(w^p)$.
If $W$ is crystallographic and $p$ is coprime to $h$, then $\sf{V}_p \simeq \sf{V}$.

For any integer $k \geq 0$, we set $[k]_\X = 1 + \X + \cdots + \X^{k - 1}$.
Generalizing the formula in \S\ref{subsec:parking-intro} for the crystallographic case, we define the \dfemph{rational parabolic $\X$-parking numbers} of $(W, p, J)$ to be
\begin{align}\label{eq:parking-twist}
\Park_{W, p}^{J, \pm}(\X)
	= \prod_{i = 1}^{|J|} \frac{[p \pm e_i^{J, p}]_\X}{[d_i^J]_\X},
\end{align}
where $d_1^J, \ldots, d_{|J|}^J$ are the fundamental degrees of $W_J$, and $e_1^{J, p},\ldots, e_{|J|}^{J, p}$ are the \dfemph{exponents} or \dfemph{fake degrees} of the $W_J$-action on $\sf{V}_p^\ast$, as defined in \cite{br}.

Recall that a Coxeter word in $S$ is a word $\vec{c}$ formed by placing the elements of $S$ in any order.
We write $\vec{c}^p$ for the concatenation of $p$ copies of $\vec{c}$.
The goal of this section is the following identity, which implies \Cref{cor:parking} in the $\X \to 1$ limit.

\begin{thm}\label[thm]{thm:parking-q}
If $W$ is crystallographic, then for any Coxeter word $\vec{c}$ in $S$, integer $p > 0$ coprime to $h$, and subset $J \subseteq S$, we have
\begin{align}
\Park_{W, p}^{J, \pm}(\X)
	&= \frac{1}{(\X - 1)^r}
	\sum_{v \in W^{J, \mp}} 
	\sum_{\vec{\omega} \in \cal{D}^{(v)}(\vec{c}^p)}
	\X^{|\sf{d}_{\vec{\omega}}|} 
		(\X - 1)^{|\sf{e}_{\vec{\omega}}|}.
\end{align}
(Note the sign flip.)
\end{thm}

\begin{conj}
\Cref{thm:parking-q} generalizes to any irreducible finite Coxeter system when $\Park_{W, p}^{J, \pm}(\X)$ is defined using \eqref{eq:parking-twist}.
\end{conj}

\subsection{From Products to Traces}

We continue to allow non-crystallographic $W$.
Let $\bb{K}$ be a splitting field for $W$, so that $\sf{V}_p$ is defined over $\bb{K}$.
When $W$ is crystallographic, we can take $\bb{K} = \bb{Q}$.

There is a graded representation $\sf{L}_{p/h} = \bigoplus_i \sf{L}_{p/h}^i$ of $W$ that may be called the \dfemph{rational parking space} for $(W, p)$, in the spirit of \cite{arr, alw}, as its graded dimension is $[p]_\X^r$.
Explicitly, $\sf{L}_{p/h}$ is the representation of $W$ underlying the simple spherical module of the rational Cherednik algebra of $W$ at parameter $p/h$, equipped with a shift of the $W$-stable grading arising from the Euler element.

We view the graded character of $\sf{L}_{p/h}$ as a $\bb{K}[\X]$-valued trace on $\bb{K} W$.
To describe it explicitly, let $\sf{S} = \bigoplus_i \sf{S}^i$ and $\bigwedge_p = \bigoplus_j \bigwedge_p^j$, where 
\begin{align} 
	\sf{S}^i \vcentcolon= \Sym^i(\sf{V}^\ast)
	\quad\text{and}\quad
	\textstyle\bigwedge_p^j \vcentcolon= 	\textstyle\bigwedge^j(\sf{V}_p^\ast).
\end{align} 
Then for all $w \in W$, we have
\begin{align}\label{eq:rca-det}
\sum_i \X^i \tr(w \mid \sf{L}_{p/h}^i)
	&= \left.\left[
		\sum_{i, j} \X^i t^j \tr(w \mid \sf{S}^i \otimes \textstyle\bigwedge_p^j)
		\right]\right|_{t \to -\X^p}\\
	&= \frac{\det(1 - \X^p w \mid \sf{V}_p^\ast)}{\det(1 - \X w \mid \sf{V}^\ast)}.
\end{align}
This formula arises from a so-called BGG-resolution of $\sf{L}_{p/h}$ by Verma modules for the rational Cherednik algebra, whose underlying $W$-representations take the form $\sf{S} \otimes \bigwedge^j$.

\begin{prop}\label[prop]{prop:parking-to-rca}
For any integer $p > 0$ coprime to $h$ and subset $J \subseteq S$, we have
\begin{align}\label{eq:parking-to-rca}
	\Park_{W, p}^{J, \pm}(\X)
	= \sum_i \X^i \tr(e_{J, \pm} \mid \sf{L}_{p/h}^i).
\end{align} 
\end{prop}

\begin{proof}
We only do the $+$ case, as the $-$ case is similar.

Set $\sf{U} = \sf{V}_p$.
Using the reflecting hyperplanes for $S$, we can decompose the $W_J$-action on $\sf{V}$ as a direct sum $\sf{V} \simeq \sf{V}_J \oplus \sf{V}_J^\intercal$, where $\sf{V}_J^\intercal$ is a $(r - |J|)$-fold power of the trivial representation.
Applying the Galois twist and grading shift that take $\sf{V}$ to $\sf{U}(-p)$, we get a direct sum $\sf{U}(-p) \simeq \sf{U}_J(-p) \oplus \sf{U}_J^\intercal(-p)$, where $\sf{U}_J(-p) \simeq (\sf{V}_J)_p(-p)$ and $\sf{U}_J^\intercal(-p)$ remains a $(r - |J|)$-fold power of the trivial representation.

Therefore, the fake degrees for $\sf{U}(-p)$ as a representation of $W_J$ are formed by taking the $|J|$ fake degrees for $(\sf{V}_J)_p$, appending $r - |J|$ zeroes, and shifting everything up by $p$.
In particular, $\sf{U}(-p)$ satisfies the hypothesis in Theorem 3.1 and Corollary 3.2 of \cite{os} that the sum of the fake degrees is equal to the fake degree for its $r$th exterior power.
We deduce that $(\sf{S} \otimes \bigwedge \sf{U}(-p))^{W_J}$ remains isomorphic to an exterior algebra over $\sf{S}^{W_J}$.
So we arrive at the formula
\begin{align}
\sum_{i, j} \X^i t^j 
	\dim {(\sf{S}^i \otimes \textstyle\bigwedge_p^j)^{W_J}}
	 = \prod_i \frac{1 + t\X^{p + e_i^J}}{1 - \X^{d_i^J}},
\end{align}
which gives the desired product formula at $t \to -1$.
\end{proof}

\begin{ex}
Taking $J = \emptyset$ and $J = S$ in \Cref{prop:parking-to-rca}, we recover the formulas
\begin{align}
\Cat_{W, p}(\X)
	= \sum_i \X^i \dim {(\sf{L}_{p/h}^i)^W}
	\quad\text{and}\quad
	[p]_\X^r = \sum_i \X^i \dim {\sf{L}_{p/h}^i},
\end{align}
respectively.
\end{ex} 

\subsection{From Traces to Cells}\label{subsec:trace-to-cell}

Recall the notation $T_{\vec{c}} \in H_W$ from \S\ref{subsec:traces-as-counts}.
In \cite{trinh}, the first author showed that the value at $T_{\vec{c}}$ of the trace on $H_W$ corresponding to $\tau_G$ is the graded character of $\sf{L}_{p/h}$ up to a shift.
In our notation, this is the identity
\begin{align}\label{eq:tau-g-to-rca}
\tau_G(w \otimes T_{\vec{c}}^p)
	= \sum_i \X^i \tr(w \mid \sf{L}_{p/h}^i).
\end{align}
Now \emph{assume that $W$ is crystallographic}.
Pick split semisimple $G$ with Weyl group $W$.
In this case, 
\begin{align}
\Park_{W, p}^{J, \pm}(\X)
	&= \sum_i \X^i \tr(e_{J, \pm} \mid \sf{L}_{p/h}^i).
		&&\text{by \Cref{prop:parking-to-rca}}\\
	&= \tau_G(e_{J, \pm} \otimes T_{\vec{c}}^p)
		&&\text{by \eqref{eq:tau-g-to-rca}}\\
	&= \frac{1}{(\X - 1)^r} 
		\sum_{v \in W^{J, \pm}}
				\X^{-\ell(v)} \tau(T_{\vec{c}}^p T_{v^{-1}} T_{v\vphantom{^{-1}}})
		&&\text{by \Cref{thm:trace}}.
\end{align}
Applying \eqref{eq:gltw-generic} to the last expression, we get \Cref{thm:parking-q}.

\section{Markov Traces and Kirkman Numbers}\label{sec:markov}

\subsection{}

In this section, we prove \Cref{thm:markov} and \Cref{cor:kirkman}.
Along the way, we review Markov traces, the HOMFLYPT polynomial, and rational Kirkman polynomials.
Unless otherwise specified, $W = S_n$ and $S = \{s_1, \ldots, s_{n - 1}\}$, as in \S\ref{subsec:lascoux}.
\subsection{Markov Traces and HOMFLYPT}

As explained in \cite{jones-v} (in a different normalization), there is a unique family of traces 
\begin{align} 
\mu_n : H_{S_n} \to \bb{Q}(\X^{1/2})[a^{\pm 1}]
\end{align} 
satisfying these conditions:
\begin{enumerate}
\item 	$\mu_1(1) = 1$.

\item 	For all $\beta \in H_{S_{n - 1}}$, we have
		\begin{align} 
		\mu_{n + 1}(\beta T_{s_n}^{\pm 1}) = (-a^{-1}\X^{1/2})^{\pm 1}\, \mu_n(\beta).
		\end{align} 
		In particular, $\mu_{n + 1}(\beta) = \dfrac{a - a^{-1}}{\X^{1/2} - \X^{-1/2}}\, \mu_n(\beta)$, due to the quadratic relation on $T_{s_n}$.

\end{enumerate}
These traces give rise to an isotopy invariant of (tame) topological links.

Namely:
Any topological braid on $n$ strands $\beta$ defines an element of $H_{S_n}$, which we again denote by $\beta$, via the map from the braid group to $H_{S_n}$ that sends the $i$th positive simple twist $\sigma_i$ to the element $\X^{-1/2} T_{s_i}$.
For instance, if $\vec{s} = (s_{i_1}, \ldots, s_{i_\ell})$, then this map sends the positive braid $\sigma_{i_1} \cdots \sigma_{i_\ell}$ to the element $\X^{-\ell} T_{\vec{s}}$.
At the same time, closing up $\beta$ by wrapping it around a solid torus, then embedding it into $3$-space, defines a link $\hat{\beta}$ up to isotopy, called the \dfemph{closure} of $\beta$.
Ocneanu showed that if $e(\beta) \in \bb{Z}$ is the \dfemph{writhe} of $\beta$, meaning its length with respect to positive simple twists, then 
\begin{align}
\bb{P}(\hat{\beta}) \vcentcolon= (-a)^{e(\beta)} \mu_n(\beta) \in \bb{Q}(\X^{1/2})[a^{\pm 1}]
\end{align}
only depends on $\hat{\beta}$.

The Laurent polynomial $\bb{P}(\hat{\beta})$ is now called its \dfemph{reduced HOMFLYPT polynomial}, after its discoverers.
(The ``O'' stands for Ocneanu; the adjective ``reduced'' means that the normalization satisfies $\bb{P}(\text{unknot}) = 1$.)
The traces $\mu_n$ are called \dfemph{Markov traces}, as condition (2) in their definition corresponds to the so-called second Markov move on braids.
For further details, see \cite{jones-v}.

In \cite{gomi}, Y.\ Gomi introduced a uniform generalization of the traces $\mu_n$ to finite Coxeter groups $W$.
In \cite{ww}, Webster--Williamson gave a construction of Gomi's traces from weight filtrations on the cohomology of mixed sheaves.
Building on their work, the main result of \cite{trinh} relates a categorification of Gomi's traces to a Springer action of $W$ on the weight-filtered, $\bb{G}$-equivariant cohomology of the Steinberg varieties $\bb{Z}_\emptyset^-(\vec{s}) = \bb{Z}_\emptyset^+(\vec{s})$.

\subsection{Individual $a$-Degrees}

Induction on $|e(\beta)|$ shows that if $\beta \in H_{S_n}$ arises from a topological braid, then the only exponents of $a$ that can occur in $\mu(\beta)$ are
\begin{align} 
-n + 1,\quad
-n + 3,\quad 
\ldots,\quad 
n - 1.
\end{align} 
For $0 \leq k \leq n - 1$, we define $\mu_n^{(k)} : H_{S_n} \to \bb{Q}(\X^{1/2})$ by
\begin{align} 
\mu_n^{(k)}(\beta) = \text{$\bb{Q}(\X^{1/2})$-coefficient of $a^{-n + 1 + 2k}$ in $\mu_n(\beta)$}.
\end{align} 
By linearity, this is still a trace.

When $G$ is (split) semisimple of type $A_{n - 1}$, the formula for categorified traces in \cite{trinh} decategorifies to a formula relating $\mu_n^{(k)}$ to $\tau_G$.
To state it, let $e_{\wedge^k} \in \bb{Q}S_n$ be the symmetrizer for the $k$th exterior power of the reflection representation $\sf{V} \simeq \sf{V}^\ast$.
For any finite, irreducible Coxeter group $W$ of rank $r$, such elements $e_{\wedge^k} \in \bb{Q}W$ may be defined for $0 \leq k \leq r$ through the formal identity
\begin{align}\label{eq:exterior}
\frac{1}{|W|}
	\sum_{w \in W} \det(1 - t w \mid \sf{V})w
	= \sum_{k = 0}^r
		{(-t)^k} e_{\wedge^k}.
\end{align}
Note that $e_{\wedge^0} = e_{S, +}$ and $e_{\wedge^{n - 1}} = e_{S, -}$, in the notation of \S\ref{subsec:symmetrizer}.
For $G$ (split) semisimple of type $A_{n - 1}$, we have:
\begin{align}\label{eq:markov-to-tau-g}
\mu_n^{(k)} = (\X - 1)^{n - 1}\,
	\tau_G(e_{\wedge^k} \otimes \whitearg).
\end{align}
The proof amounts to plugging $z = e_{\wedge^k}$ into \eqref{eq:tau-g-to-exotic}, then rearranging terms using \eqref{eq:exterior} to arrive at the character-theoretic formula for $\mu_n^{(k)}$ in \cite[\S{4.3}]{gomi}.

Meanwhile, in \cite{bt}, Bezrukavnikov--Tolmachov gave a formula that (in our normalization) relates $\mu_n^{(k)}$ to $\mu_n^{(n - 1)}$.
To state it, we need the \dfemph{multiplicative Jucys--Murphy elements} $\JM_k \in H_{S_n}$ defined by
\begin{align}
\JM_k = \X^{-k} T_{s_k \cdots s_2s_1} T_{s_1s_2 \cdots s_k}
	\quad\text{for $1 \leq k \leq n$}.
\end{align}
Let $e_i(X_1, \ldots, X_{n - 1})$ be the elementary symmetic polynomial of degree $i$ in variables $X_1, \ldots, X_{n - 1}$.
Then \cite[Cor.\ 6.1.2]{bt} is the identity
\begin{align}\label{eq:bt-original}
\mu_n^{(k)}(\beta)
	= \mu_n^{(n - 1)}(\beta e_{n - 1 - k}(\JM_1, \ldots, \JM_{n - 1})).
\end{align}
It turns out that $\mu_n^{(n - 1)}$ is precisely the trace denoted $\tau$ in \S\ref{subsec:trace-intro}, as one can also deduce from \eqref{eq:markov-to-tau-g} and \Cref{thm:trace}.

\begin{rem}
Note that the variable $a$ in \cite[\S{6}]{bt} is our variable $-a^{-2}$, up to an overall grading shift.
Hence their $\ur{Tr}_n^{(k)}$ is our $\mu_n^{(n - 1 - k)}$, \emph{etc.}
\end{rem}

\begin{rem}
Jucys--Murphy elements were originally defined in the context of the group rings $\bb{Z}S_n$.
One can show \cite[(3)]{io} that
\begin{align}
	\frac{\JM_k - 1}{\X^{1/2} - \X^{-1/2}}
	= \sum_{i = 1}^k
	\X^{i - k}
	T_{s_i \cdots s_{k - 1}} T_{s_k} T_{s_{k - 1} \cdots s_i}.
\end{align}
At $\X \to 1$, the right-hand side specializes to the $k$th classical Jucys--Murphy element in $\bb{Z}S_n$.
These elements generate a maximal commutative subalgebra of $\bb{Z}S_n$.
Similarly, the $\JM_k$ generate a maximal commutative subalgebra of $H_{S_n}$ \cite[Prop.\ 1]{io}.
\end{rem}

\subsection{Jucys--Murphy Products}\label{subsec:jm}

Recall that $\Asc(v)$ and $\Des(v)$ respectively denote the left ascent and descent sets of $v$.
From \eqref{eq:bt-original}, we reduce \Cref{thm:markov} to:

\begin{thm}\label[thm]{thm:jm}
For all $k$, we have 
\begin{align} 
e_{n - 1 - k}(\JM_1, \ldots, \JM_{n - 1})
	= \sum_{\Des(v) = I_k} \X^{-\ell(v)} T_{v^{-1}} T_{v\vphantom{^{-1}}},
\end{align} 
where $I_k = \{s_1, \ldots, s_{n - 1 - k}\} \subseteq S$.
\end{thm}

\begin{ex}
Taking $k = 0$ above, we get
\begin{align} 
	\JM_1 \cdots \JM_{n - 1} = \X^{-\ell_S} T_{w_\circ}^2.
\end{align} 
Through this identity, \eqref{eq:bt-original} implies that the ``lowest'' and ``highest'' $a$-degrees of $\mu_n$ are related by the \dfemph{full twist} $\Delta^2 \vcentcolon= \X^{-\ell_S} T_{w_\circ}^2$: explicitly,
\begin{align}
	\mu_n^{(0)}(\beta) = \mu_n^{(n - 1)}(\beta \Delta^2),
\end{align}
an identity originally discovered by K\'alm\'an \cite{kalman}.
Compare to \Cref{rem:stz}.
\end{ex}

The proof of \Cref{thm:jm} amounts to \Crefrange{lem:jm-to-c}{lem:c-to-des} below.
As preparation, for any subset $I = \{s_{i_1}, \ldots, s_{i_j}\} \subseteq S$, let
\begin{align}
	\begin{array}{r@{\:}l@{\:}l@{\quad}l}
\JM(I) &= \JM_{i_1, \ldots, i_j}
	&\vcentcolon= \textstyle\prod_i^\downarrow
			\JM_i &\in H_{S_n},\\
c(I) &= c_{i_1, \ldots, i_j}
	&\vcentcolon= \textstyle\prod_i^\downarrow
			{(s_1 \cdots s_i)} &\in S_n,
\end{array}
\end{align}
where the notation $\prod_i^\downarrow$ means the product over $i_1, \ldots, i_j$ in decreasing order.

\begin{lem}\label[lem]{lem:jm-to-c}
For any subset $I \subseteq S$, we have 
\begin{align} 
\JM(I) = \X^{-\ell(c(I))} T_{c(I)^{-1}} T_{c(I)\vphantom{^{-1}}}.
\end{align} 
\end{lem} 

\begin{proof}
Let $i_1 < i_2 < \cdots < i_j$ be the elements of $I$.
For any $i, k$ with $1 \leq k < i \leq n - 1$, we have the relations
\begin{align}
T_{s_k} T_{s_i \cdots s_2s_1}
	= T_{s_i \cdots s_2s_1} T_{s_{k+ 1}}
	\quad\text{and}\quad
T_{s_k} T_{s_1s_2 \cdots s_i}
	= T_{s_1s_2 \cdots s_i} T_{s_{k - 1}},
\end{align}
as one can check from braid diagrams.
Using these relations, we can move the prefixes $T_{s_1s_2 \cdots s_i}$ in each factor $\JM_{i_k}$ of $\JM(I)$ from right to left, through each of $\JM_{i_{k + 1}}$, \ldots, $\JM_{i_j}$, giving the result.
\end{proof}

\begin{ex}
In what follows, we write $T_i$ in place of $T_{s_i}$ for convenience.
When $n = 4$ and $|I| = 2$, we have
\begin{align}
\JM_{1, 2}
	&= \X^{-3}
		T_2 T_1^2 T_2 \cdot T_1^2,
	&c_{1, 2}
		&= s_1s_2 \cdot s_1,\\
\JM_{1, 3}
	&= \X^{-4}
		T_3 T_2 T_1^2 T_2 T_3 \cdot T_1^2,
	&c_{1, 3}
		&= s_1s_2s_3 \cdot s_1,\\
\JM_{2, 3}
	&= \X^{-5}
		T_3 T_2 T_1^2 T_2 T_3 \cdot T_2 T_1^2 T_2,
		&c_{2, 3}
		&= s_1s_2s_3 \cdot s_1s_2,
\end{align}
\Cref{lem:jm-to-c} says that
\begin{align}
\JM_{1, 2}
	&= \X^{-3}
		T_1T_2T_1 \cdot T_1T_2T_1,\\
\JM_{1, 3}
	&= \X^{-4}
		T_1 \cdot T_3T_2T_1 \cdot T_1T_2T_3 \cdot T_1,\\
\JM_{2, 3}
	&= \X^{-5}
		T_2T_1 \cdot T_3T_2T_1 \cdot T_1T_2T_3 \cdot T_1T_2.
\end{align}
\end{ex} 

\begin{lem}\label[lem]{lem:c-to-des}
For $1 \leq j \leq n$. we have
\begin{align}
\{c(I) \mid |I| = j-1\}
	= \{v \in S_n \mid \Des(v) = \{s_1, \ldots, s_{j-1}\}\}.
\end{align}
\end{lem}

\begin{proof}
Let $J = \{s_1, \ldots, s_{j-1}\}$ and $J' = S \setminus \{s_{j}\}$ in what follows.
We claim that any element $v \in S_n$ with $\Des(v) = J$ must take the form $w_{J\circ}v'$, where $v'$ is a minimal-length right coset representative of $W_{J'} \simeq S_j \times S_{n - j}$.
Indeed, $\Des(v) \supseteq J$ forces $w_{J\circ}$ to be a left factor of $v$, and if $v = w_{J\circ}v'$, then the reverse inclusion $\Des(v) \subseteq J$ forces the condition on $v'$.

Note that there are exactly $\binom{n}{j}$ elements of the form $w_{J\circ}v'$ with $v' \in W^{J', -}$.  These are exactly the elements $c(I)$ with $|I| = j$.
This calculation shows that $\prod_i^\downarrow {(s_1 \cdots s_i)}$ is reduced and that $\Des(c(I)) = J$.
As there are $\binom{n}{j}$ choices for $I$ such that $|I| = j$, the corresponding elements $c(I)$ exhaust the elements $v$ such that $\Des(v) = J$.
\end{proof}

\begin{cor}\label[cor]{cor:markov}
	For any word $\vec{s}$ in $S = \{s_1, \ldots, s_{n - 1}\}$, we have
	\begin{align}
		\mu_n^{(k)}(T_{\vec{s}})
		= \frac{1}{(\X - 1)^{n - 1}}
		\sum_{\Asc(v) = I_k}
		\sum_{\vec{\omega} \in \cal{D}^{(v)}(\vec{s})}
		\X^{|\mathbf{d}_{\vec{\omega}}|}
		(\X - 1)^{|\mathbf{e}_{\vec{\omega}}|}.
	\end{align}
\end{cor}

\begin{proof}
	Combine \eqref{eq:bt-original}, \Cref{thm:jm}, and \eqref{eq:gltw-generic} to arrive at a double sum over $v$ such that $\Des(v) = I_k$ and over $\vec{\omega}$ in $\cal{D}^{(vw_\circ)}(\vec{s})$.
	Then note that $\ell(sv) < \ell(v)$ if and only if $\ell(svw_\circ) > \ell(svw_\circ)$.
\end{proof}

\subsection{Kirkman Numbers}

For any finite, irreducible Coxeter group $W$ of rank $r$ and Coxeter number $h$, and integer $p > 0$ coprime to $h$, we define the \dfemph{rational Kirkman polynomials} of $(W, p)$ to be
\begin{align}
	\Kirk_{W, p}^{(k)}(\X)
	= \frac{\det(1 - \X^p e_{\wedge^k} \mid \sf{V}_p^\ast)}{\det(1 - \X e_{\wedge^k} \mid \sf{V}^\ast)}
	\quad\text{for $0 \leq k \leq r$}.
\end{align}
Equivalently, by \eqref{eq:exterior},
\begin{align}
	\sum_{k = 0}^r t^k\, \Kirk_{W, p}^{(k)}(\X)
	= \frac{1}{|W|} \sum_{w \in W} 
	\frac{\det(1 + tw \mid \sf{V}^\ast) \det(1 - \X^p w \mid \sf{V}_p^\ast)}{\det(1 - \X w \mid \sf{V}^\ast)}.
\end{align}
When $p = h + 1$, this definition recovers the \dfemph{Kirkman polynomials} of $W$ introduced in \cite[\S{9.2}]{arr}.

We define the \dfemph{rational Kirkman numbers} of $(W, p)$ by $\Kirk_{W, p}^{(k)} \vcentcolon= \Kirk_{W, p}^{(k)}(1)$.
For $W = S_n$ and $p = n + 1$, they recover the $f$-vectors of the usual associahedron \cite{arw}.
For general $W$ and $p = h + 1$, they recover the $f$-vectors of the $W$-associahedron in \cite{fr}, as noted in \cite[\S{3.3}]{arr}.
We discuss this further in \Cref{sec:associahedron}.

In the $\X \to 1$ limit, the following identity implies \Cref{cor:kirkman} about the rational Kirkman numbers of $S_n$.
\Cref{fig:a3p3} at the end of the paper illustrates \Cref{cor:parking} and \Cref{cor:kirkman} simultaneously.

\begin{thm}\label[thm]{thm:kirkman}
	Take $W = S_n$ and $S = \{s_1, \ldots, s_{n - 1}\}$.
	Then for any Coxeter word $\vec{c}$, integer $p > 0$ coprime to $n$, and integer $k$, we have
	\begin{align}
		\Kirk_{S_n, p}^{(k)}(\X)
		= \frac{1}{(\X - 1)^{n - 1}}
		\sum_{\Asc(v) = I_k}
		\sum_{\vec{\omega} \in \cal{D}^{(v)}(\vec{c}^p)}
		\X^{|\sf{d}_{\vec{\omega}}|} 
		(\X - 1)^{|\sf{e}_{\vec{\omega}}|}.
	\end{align}
\end{thm}

\begin{proof}
	Pick any split semisimple $G$ of type $A_{n - 1}$.
	Then
	\begin{align}
		\Kirk_{S_n, p}^{(k)}(\X)
		&= \tau_G(e_{\wedge^k} \otimes T_{\vec{c}}^p)
		&&\text{by \eqref{eq:rca-det} and \eqref{eq:tau-g-to-rca}}\\
		&= \frac{1}{(\X - 1)^{n - 1}}\, \tau[\zeta_{I_k}^+](T_{\vec{c}}^p)
		&&\text{by \Cref{thm:markov}}\\
		&= \frac{1}{(\X - 1)^{n - 1}}
		\sum_{\Des(v) = I_k}
		\X^{-\ell(v)}\,
		\tau(T_{\vec{c}}^p T_{v^{-1}} T_{v\vphantom{^{-1}}}).
	\end{align}
	Apply \eqref{eq:gltw-generic} to the sum over $v$ such that $\Des(v) = I_k$.
	Then conclude as in the proof of \Cref{cor:markov}.
\end{proof}

\subsection{Other Types?}\label{subsec:other-types}

It is natural to seek generalizations of \Cref{cor:markov} and \Cref{thm:kirkman} to other $W$.
So far, we have been unable to find such a construction.
This may be related to the absence of uniform formulas for Kirkman polynomials in general.
Attractive formulas do exist for \dfemph{coincidental types}, where the degrees of $W$ form an arithmetic sequence \cite[\S{10}]{rss}.

In recent work \cite{tz}, Tolmachov--Zhylinskyi generalize the multiplicative Jucys--Murphy formula in \cite[Cor.\ 6.1.2]{bt} to types $BC$ and $D$.
It would be interesting to extend \Cref{thm:jm} to these types.
Note that type $BC$ is always coincidental, whereas type $D_n$ is not coincidental for $n \geq 4$.

\appendix
\section{Faces of the Associahedron} \label[appendix]{sec:associahedron}

We keep the notation of~\Cref{sec:markov}.
\Cref{cor:kirkman} shows that for any coprime integers $n, p > 0$, Coxeter word $\vec{c}$ of $S_n$, and integer $k$, the collection of parabolic parking objects
\begin{align}\label{eq:kirkman-family}
\coprod_{\Asc(v) = I_k}
\cal{M}^{(v)}(\vec{c}^p),
\quad\text{where $I_k = \{s_1, \ldots, s_{n - 1 - k}\}$},
\end{align}
may be viewed as a $\vec{c}$-noncrossing set for the rational Kirkman number $\Kirk_{S_n, p}^{(k)}$.
Below, in the case where $p = n + 1$, we relate these objects to the classical noncrossing combinatorics of associahedra.

For any irreducible, finite Coxeter system $(W, S)$, let $T \subseteq W$ be the set of all reflections.
Let $\leq_T$ denote the absolute order on $W$.
For any Coxeter word $\vec{c}$ representing a Coxeter element $c \in W$, let $\vec{w}_\circ(c)$ be the \dfemph{$c$-sorting word} for the longest element $w_\circ = w_{S\circ}$: that is, the first subword of $\vec{c}^p$, in lexicographical order, that is a reduced word for $w_\circ$.

Let $\Asso(W, c)$ be the \dfemph{associahedron} of~\cite{reading} and~\cite{hlt}.
Following~\cite{cls,ps}, we may view it as the simplicial complex whose faces are the subwords of $\vec{c} \vec{w}_\circ(c)$ for which the complement contains a reduced word for $w_\circ$.
The vertex set of this complex is in bijection with the set of \dfemph{$c$-noncrossing partitions}
\begin{align} 
\NC(W, c) \vcentcolon= \{ \pi \in W \mid \pi \leq_T c \}.
\end{align}
We will identify the vertices with the $c$-noncrossing partitions themselves.

For each $\pi \in \NC(W, c)$, let $W_\pi$ be the (not-necessarily-standard) parabolic subgroup of $W$ generated by the reflections $t \in T$ such that $t \leq_T \pi$.
Let $W^\pi$ be the set of minimal left coset representatives for $W_\pi$.
The \dfemph{$c$-noncrossing parking functions} of $W$ are the cosets $vW_\pi$ as we run over all $\pi \in \NC(W, c)$ and $v \in W^\pi$.

\begin{ex}
Take $W = S_3$ and $c = st$, where $s = s_1$ and $t = s_2$.
Then $\NC(W, c) = \{e, s, t, sts, c\}$.
The 16 noncrossing parking functions are
\begin{align}
&eW_e, sW_e, tW_e, stW_e, tsW_e, stsW_e,\\
&eW_s, tW_s, stW_s, \quad eW_t, sW_t, tsW_t, \quad eW_{sts}, sW_{sts}, tW_{sts},\\
&eW_c,
\end{align}
when written using minimal coset representatives.
\end{ex}

Recall that the edges out of a given vertex $\pi \in \NC(W, c)$ are indexed by the reflections in the canonical factorization of the Kreweras complement $c\pi^{-1}$.
So the largest dimension among faces of $\Asso(W, c)$ with minimal vertex $\pi$ is $r - \ell_T(\pi)$, where $\ell_T$ denotes absolute length.
More generally, the number of $k$-faces of $\Asso(W, c)$ with minimal vertex $\pi$ is $\binom{r - \ell_T(\pi)}{k}$.

The following result shows how to construct a set of this size using $c$-noncrossing parking functions, when $W = S_n$.
Note that here, $r = n - 1$.

\begin{prop}\label[prop]{prop:noncrossing}
For all $\pi \in \NC(S_n, c)$, we have
\begin{align}
	|W^\pi \cap {\sf A}_k| = \binom{n - 1 - \ell_T(\pi)}{k},
	\quad\text{where ${\sf A}_k \vcentcolon= \{w \in W \mid \Asc(w) = I_k\}$}.
\end{align} 
We therefore have a bijection from the set of $k$-faces of $\Asso(S_n, c)$ with minimal vertex $\pi$ to the set $S_n^\pi \cap {\sf A}_k$.
\end{prop}

\begin{proof} 
Without loss of generality, take $\vec{c} = (s_1, s_2, \ldots, s_r)$.
Write $w_j = s_r s_{r - 1} \cdots s_j$, so that $\sf{A}_k$ consists of the $\binom{r}{k}$ elements of the form $w_{j_1}w_{j_2} \cdots w_{j_i}$ as we run over $k$-subsets $\{ j_1 < j_2 < \cdots < j_k \}$ of $\{1, \ldots, r\}$.
(See~\Cref{ex:asc} below.)

In one-line notation, these permutations are obtained by filling in the numbers $1, 2, \ldots, n - k$ from left to right but skipping the entries in positions $j_1, \ldots, j_k$, then filling in the skipped entries from right to left with the numbers $n - k + 1, \ldots, n$.
Since $1, \ldots, n - k$ appear in order but $n - k + 1, \ldots, n$ do not, the left ascent set of such a permutation is exactly $I_k$.

For our choice of $c$, we can picture each $c$-noncrossing partition $\pi$ as a (classical) noncrossing partition of the set $[n] \vcentcolon= \{1, \ldots, n\}$.
We can picture a $c$-noncrossing parking function of the form $vW_\pi$ as a way to decorate each block of $\pi$ by a block of the same size in an arbitrary set partition of $[n]$.
It remains to show that there are $\binom{r - \ell_T(\pi)}{k}$ elements $v$ in $\sf{A}_k$ such that the letters in each block of $\pi$ appear in increasing order in the one-line notation of $v$, as this will imply that $v$ is a minimal left coset representative for $W_\pi$.

Observe that $n - \ell_T(\pi) = r - \ell_T(\pi) + 1$ is the number of blocks of $\pi$.
List the blocks as $B_1, B_2, \ldots, B_{n - \ell_T(\pi)}$, in increasing order of their largest elements.
Write these largest elements in order as $b_1, b_2, \ldots, b_{n - \ell_T(\pi)}$, so that $b_{n - \ell_T(\pi)} = n$.
Then the desired elements $v \in \sf{A}_k$ are the elements $w_{b_{j_1}}w_{b_{j_2}} \cdots w_{b_{j_k}}$ as we run over subsets $\{j_1 < j_2 < \cdots < j_k\}$ of $\{1, \ldots, r - \ell_T(\pi)\}$.
(Again, see \Cref{ex:asc}.)
\end{proof}

\begin{ex}\label[ex]{ex:asc}
Taking $W = S_3$ gives
\begin{align}
	\sf{A}_0 = \{e\},\quad 
	\sf{A}_1 = \{w_1, w_2\},\quad 
	\sf{A}_2 = \{w_1w_2\},
\end{align}
where $w_1 = s_2s_1$ and $w_2 = s_2$.
Taking $W = S_4$ gives
\begin{align}
	\sf{A}_0 = \{e\},\quad 
	\sf{A}_1 = \{w_1, w_2, w_3\},\quad
	\sf{A}_2 = \{w_1w_2, w_1w_3, w_2w_3\},\quad 
	\sf{A}_3 = \{w_1w_2w_3\},
\end{align}
where $w_1 = s_3s_2s_1$ and $w_2 = s_3s_2$ and $w_3 = s_3$.

Note that the elements $w_j$ are analogous to the elements $c(S \setminus \{s_j\})$ in \Cref{sec:markov}, but with $\Asc$ in place of $\Des$.
\end{ex}

\begin{cor}
For any integer $k$, we have a bijection from the set of $k$-faces of $\Asso(S_n, c)$  to the set of $c$-noncrossing parking functions $vW_\pi$ with $v \in S_n^\pi \cap {\sf A}_k$.
\end{cor} 

In \cite{gltw}, for any $W$ and $v \in W$, we gave a bijection from the set of $c$-noncrossing parking functions $\{vW_\pi \mid \pi \in \NC(W, c)\}$ to $\cal{M}^{(v)}(\vec{c}^{h + 1})$.
We deduce:

\begin{cor}
For any integer $k$, we have a bijection from the set of $k$-faces of $\Asso(S_n, c)$  to the set in \eqref{eq:kirkman-family}.
\end{cor} 


\bibliographystyle{alphaurl}
\bibliography{norms}
\mbox{}\\[-3ex]

\newcommand*{\captionfig}{\protect\scalebox{0.125}{\begin{tikzpicture}
\protect\fill (2,0) rectangle ++ (1,1); 
\protect\fill (1,1) rectangle ++ (1,1); 
\protect\fill (0,2) rectangle ++ (1,1); 
\protect\draw (0,0) grid (3,3);
\end{tikzpicture}}}

\begin{figure}[htb]
\scalebox{0.625}{
	\begin{tikzpicture}
		\node[draw] (0) at (0,0) {\scalebox{.8}{$27 \hspace{1em} J=\emptyset  \hspace{1em} 0$}};
		\node[draw] (3) at (-5,2.5) {\scalebox{.8}{
				\begin{tikzpicture}[anchor=south]
					\node (121a) at (0,1) {\scalebox{.5}{
							\begin{tikzpicture}
								\fill (2,0) rectangle ++ (1,1); 
								\fill (2,1) rectangle ++ (1,1); 
								\fill (2,2) rectangle ++ (1,1); 
								\draw (0,0) grid (3,3);
								\node (n) at (1.5,-1) {\scalebox{2}{$s_1s_2s_1$}};
					\end{tikzpicture}}};
					\node (txt) at (0,0) {$18 \hspace{1em} J=\{3\} \hspace{1em} 1$};
		\end{tikzpicture}}};
		\node[draw] (2) at (0,2.5) {\scalebox{.8}{
				\begin{tikzpicture}[anchor=south]
					\node (13a) at (-2.2,1) {\scalebox{.5}{
							\begin{tikzpicture}
								\fill (0,0) rectangle ++ (1,1); 
								\fill (2,1) rectangle ++ (1,1); 
								\fill (1,2) rectangle ++ (1,1); 
								\draw (0,0) grid (3,3);
								\node (n) at (1.5,-1) {\scalebox{2}{$s_1s_3$}};
					\end{tikzpicture}}};
					\node (13b) at (0,1) {\scalebox{.5}{
							\begin{tikzpicture}
								\fill (1,0) rectangle ++ (1,1); 
								\fill (0,1) rectangle ++ (1,1); 
								\fill (2,2) rectangle ++ (1,1); 
								\draw (0,0) grid (3,3);
								\node (n) at (1.5,-1) {\scalebox{2}{$s_1s_3$}};
					\end{tikzpicture}}};
					\node (132) at (2.2,1) {\scalebox{.5}{
							\begin{tikzpicture}
								\fill (2,0) rectangle ++ (1,1); 
								\fill (1,1) rectangle ++ (1,1); 
								\fill (0,2) rectangle ++ (1,1); 
								\draw (0,0) grid (3,3);
								\node (n) at (1.5,-1) {\scalebox{2}{$s_1s_3s_2$}};
					\end{tikzpicture}}};
					\node (txt) at (0,0) {$18 \hspace{1em} J=\{2\} \hspace{1em} 3$};
		\end{tikzpicture}}};
		\node[draw] (1) at (5,2.5) {\scalebox{.8}{
				\begin{tikzpicture}[anchor=south]
					\node (232a) at (0,1) {\scalebox{.5}{
							\begin{tikzpicture}
								\fill (0,0) rectangle ++ (1,1); 
								\fill (0,1) rectangle ++ (1,1); 
								\fill (0,2) rectangle ++ (1,1); 
								\draw (0,0) grid (3,3);
								\node (n) at (1.5,-1) {\scalebox{2}{$s_2s_3s_2$}};
					\end{tikzpicture}}};
					\node (txt) at (0,0) {$18 \hspace{1em} J=\{1\} \hspace{1em} 1$};
		\end{tikzpicture}}};
		\node[draw] (23) at (-7.5,7.5) {\scalebox{.8}{
				\begin{tikzpicture}[anchor=south]
					\node (1a) at (-2.2,1) {\scalebox{.5}{
							\begin{tikzpicture}
								\fill (0,0) rectangle ++ (1,1); 
								\fill (2,1) rectangle ++ (1,1); 
								\fill (1,2) rectangle ++ (1,1); 
								\draw (0,0) grid (3,3);
								\node (n) at (1.5,-1) {\scalebox{2}{$s_1$}};
					\end{tikzpicture}}};
					\node (1b) at (0,1) {\scalebox{.5}{
							\begin{tikzpicture}
								\fill (1,0) rectangle ++ (1,1); 
								\fill (0,1) rectangle ++ (1,1); 
								\fill (2,2) rectangle ++ (1,1); 
								\draw (0,0) grid (3,3);
								\node (n) at (1.5,-1) {\scalebox{2}{$s_1$}};
					\end{tikzpicture}}};
					\node (1c) at (2.2,1) {\scalebox{.5}{
							\begin{tikzpicture}
								\fill (2,2) rectangle ++ (1,1); 
								\fill (2,1) rectangle ++ (1,1); 
								\fill (2,0) rectangle ++ (1,1); 
								\draw (0,0) grid (3,3);
								\node (n) at (1.5,-1) {\scalebox{2}{$s_1$}};
					\end{tikzpicture}}};
					\node (12a) at (-1.2,3.5) {\scalebox{.5}{
							\begin{tikzpicture}
								\fill (0,2) rectangle ++ (1,1); 
								\fill (1,1) rectangle ++ (1,1); 
								\fill (2,0) rectangle ++ (1,1); 
								\draw (0,0) grid (3,3);
								\node (n) at (1.5,-1) {\scalebox{2}{$s_1s_2$}};
					\end{tikzpicture}}};
					\node (12b) at (1.2,3.5) {\scalebox{.5}{
							\begin{tikzpicture}
								\fill (2,0) rectangle ++ (1,1); 
								\fill (2,1) rectangle ++ (1,1); 
								\fill (2,2) rectangle ++ (1,1); 
								\draw (0,0) grid (3,3);
								\node (n) at (1.5,-1) {\scalebox{2}{$s_1s_2$}};
					\end{tikzpicture}}};
					\node (txt) at (0,0) {$10 \hspace{1em} J=\{2,3\} \hspace{1em} 5$};
		\end{tikzpicture}}};
		\node[draw] (12) at (7.5,7.5) {\scalebox{.8}{
				\begin{tikzpicture}[anchor=south]
					\node (3a) at (2.2,1) {\scalebox{.5}{
							\begin{tikzpicture}
								\fill (0,0) rectangle ++ (1,1); 
								\fill (2,1) rectangle ++ (1,1); 
								\fill (1,2) rectangle ++ (1,1); 
								\draw (0,0) grid (3,3);
								\node (n) at (1.5,-1) {\scalebox{2}{$s_3$}};
					\end{tikzpicture}}};
					\node (3b) at (0,1) {\scalebox{.5}{
							\begin{tikzpicture}
								\fill (1,0) rectangle ++ (1,1); 
								\fill (0,1) rectangle ++ (1,1); 
								\fill (2,2) rectangle ++ (1,1); 
								\draw (0,0) grid (3,3);
								\node (n) at (1.5,-1) {\scalebox{2}{$s_3$}};
					\end{tikzpicture}}};
					\node (3c) at (-2.2,1) {\scalebox{.5}{
							\begin{tikzpicture}
								\fill (0,2) rectangle ++ (1,1); 
								\fill (0,1) rectangle ++ (1,1); 
								\fill (0,0) rectangle ++ (1,1); 
								\draw (0,0) grid (3,3);
								\node (n) at (1.5,-1) {\scalebox{2}{$s_3$}};
					\end{tikzpicture}}};
					\node (32a) at (1.2,3.5) {\scalebox{.5}{
							\begin{tikzpicture}
								\fill (0,2) rectangle ++ (1,1); 
								\fill (1,1) rectangle ++ (1,1); 
								\fill (2,0) rectangle ++ (1,1); 
								\draw (0,0) grid (3,3);
								\node (n) at (1.5,-1) {\scalebox{2}{$s_3s_2$}};
					\end{tikzpicture}}};
					\node (32b) at (-1.2,3.5) {\scalebox{.5}{
							\begin{tikzpicture}
								\fill (0,0) rectangle ++ (1,1); 
								\fill (0,1) rectangle ++ (1,1); 
								\fill (0,2) rectangle ++ (1,1); 
								\draw (0,0) grid (3,3);
								\node (n) at (1.5,-1) {\scalebox{2}{$s_3s_2$}};
					\end{tikzpicture}}};
					\node (txt) at (0,0) {$10 \hspace{1em} J=\{1,2\} \hspace{1em} 5$};
		\end{tikzpicture}}};
		\node[draw] (13) at (0,7.5) {\scalebox{.8}{
				\begin{tikzpicture}[anchor=south]
					\node (2a) at (2.2,1) {\scalebox{.5}{
							\begin{tikzpicture}
								\fill (0,0) rectangle ++ (1,1); 
								\fill (2,1) rectangle ++ (1,1); 
								\fill (1,2) rectangle ++ (1,1); 
								\draw (0,0) grid (3,3);
								\node (n) at (1.5,-1) {\scalebox{2}{$s_2$}};
					\end{tikzpicture}}};
					\node (2b) at (0,1) {\scalebox{.5}{
							\begin{tikzpicture}
								\fill (1,0) rectangle ++ (1,1); 
								\fill (0,1) rectangle ++ (1,1); 
								\fill (2,2) rectangle ++ (1,1); 
								\draw (0,0) grid (3,3);
								\node (n) at (1.5,-1) {\scalebox{2}{$s_2$}};
					\end{tikzpicture}}};
					\node (2c) at (-2.2,1) {\scalebox{.5}{
							\begin{tikzpicture}
								\fill (0,2) rectangle ++ (1,1); 
								\fill (0,1) rectangle ++ (1,1); 
								\fill (0,0) rectangle ++ (1,1); 
								\draw (0,0) grid (3,3);
								\node (n) at (1.5,-1) {\scalebox{2}{$s_2$}};
					\end{tikzpicture}}};
					\node (23a) at (3.6,3.5) {\scalebox{.5}{
							\begin{tikzpicture}
								\fill (0,2) rectangle ++ (1,1); 
								\fill (1,1) rectangle ++ (1,1); 
								\fill (2,0) rectangle ++ (1,1); 
								\draw (0,0) grid (3,3);
								\node (n) at (1.5,-1) {\scalebox{2}{$s_2s_3$}};
					\end{tikzpicture}}};
					\node (23b) at (1.2,3.5) {\scalebox{.5}{
							\begin{tikzpicture}
								\fill (0,0) rectangle ++ (1,1); 
								\fill (0,1) rectangle ++ (1,1); 
								\fill (0,2) rectangle ++ (1,1); 
								\draw (0,0) grid (3,3);
								\node (n) at (1.5,-1) {\scalebox{2}{$s_2s_3$}};
					\end{tikzpicture}}};
					\node (21a) at (-1.2,3.5) {\scalebox{.5}{
							\begin{tikzpicture}
								\fill (0,2) rectangle ++ (1,1); 
								\fill (1,1) rectangle ++ (1,1); 
								\fill (2,0) rectangle ++ (1,1); 
								\draw (0,0) grid (3,3);
								\node (n) at (1.5,-1) {\scalebox{2}{$s_2s_1$}};
					\end{tikzpicture}}};
					\node (21b) at (-3.6,3.5) {\scalebox{.5}{
							\begin{tikzpicture}
								\fill (0,0) rectangle ++ (1,1); 
								\fill (0,1) rectangle ++ (1,1); 
								\fill (0,2) rectangle ++ (1,1); 
								\draw (0,0) grid (3,3);
								\node (n) at (1.5,-1) {\scalebox{2}{$s_2s_1$}};
					\end{tikzpicture}}};
					\node (txt) at (0,0) {$12 \hspace{1em} J=\{1,3\} \hspace{1em} 7$};
		\end{tikzpicture}}};
		\node[draw] (123) at (0,12.5) {\scalebox{.8}{
				\begin{tikzpicture}[anchor=south]
					\node (ea) at (-6,1) {\scalebox{.5}{
							\begin{tikzpicture}
								\fill (0,0) rectangle ++ (1,1); 
								\fill (0,1) rectangle ++ (1,1); 
								\fill (0,2) rectangle ++ (1,1); 
								\draw (0,0) grid (3,3);
								\node (n) at (1.5,-1) {\scalebox{2}{$e$}};
					\end{tikzpicture}}};
					\node (eb) at (-3,1) {\scalebox{.5}{
							\begin{tikzpicture}
								\fill (0,0) rectangle ++ (1,1); 
								\fill (2,1) rectangle ++ (1,1); 
								\fill (1,2) rectangle ++ (1,1); 
								\draw (0,0) grid (3,3);
								\node (n) at (1.5,-1) {\scalebox{2}{$e$}};
					\end{tikzpicture}}};
					\node (ec) at (0,1) {\scalebox{.5}{
							\begin{tikzpicture}
								\fill (1,0) rectangle ++ (1,1); 
								\fill (0,1) rectangle ++ (1,1); 
								\fill (2,2) rectangle ++ (1,1); 
								\draw (0,0) grid (3,3);
								\node (n) at (1.5,-1) {\scalebox{2}{$e$}};
					\end{tikzpicture}}};
					\node (ed) at (3,1) {\scalebox{.5}{
							\begin{tikzpicture}
								\fill (2,0) rectangle ++ (1,1); 
								\fill (1,1) rectangle ++ (1,1); 
								\fill (0,2) rectangle ++ (1,1); 
								\draw (0,0) grid (3,3);
								\node (n) at (1.5,-1) {\scalebox{2}{$e$}};
					\end{tikzpicture}}};
					\node (ee) at (6,1) {\scalebox{.5}{
							\begin{tikzpicture}
								\fill (2,0) rectangle ++ (1,1); 
								\fill (2,1) rectangle ++ (1,1); 
								\fill (2,2) rectangle ++ (1,1); 
								\draw (0,0) grid (3,3);
								\node (n) at (1.5,-1) {\scalebox{2}{$e$}};
					\end{tikzpicture}}};
					\node (txt) at (0,0) {$5 \hspace{1em} J=\{1,2,3\} \hspace{1em} 5$};
		\end{tikzpicture}}};
		\draw (0.north west) -- (3.south east);
		\draw (0.north east) -- (1.south west);
		\draw (0.north) -- (2.south);
		\draw (3.north east) -- (13.south west);
		\draw (3.north west) -- (23.south west);
		\draw (1.north west) -- (13.south east);
		\draw (1.north east) -- (12.south east);
		\draw (2.north west) -- (23.south east);
		\draw (2.north east) -- (12.south west);
		\draw (12.north east) -- (123.south east);
		\draw (13.north) -- (123.south);
		\draw (23.north west) -- (123.south west);
\end{tikzpicture}
}
	\caption{\small{We take $W = S_4$ and $\vec{c} = (s_1, s_2, s_3)$ and $p = 3$.
	The box for a set $J$ consists of pairs $(v, \vec{\omega})$ such that $\Asc(v) = J$ and $\vec{\omega} \in \mathcal{D}^{(v)}(\vec{c}^p)$.
	Edges between boxes are inclusions of $J$'s.
	Each $\vec{\omega}$ is drawn as a $3 \times 3$ array, with elements of $\mathbf{e}_{\vec{\omega}}$ in black.
	For example, \raisebox{-2pt}{\captionfig} represents $\vec{\omega} = (\id, s_2, s_3, s_1, \id, s_3, s_1, s_2, \id)$.
	In each box, the number to the right, \emph{resp.}\@ left, of $J$ counts the pairs in the box, \emph{resp.}\ among boxes for supersets of $J$.
	The latter gives $\Park_{W, p}^{J, +}$.
The rightmost number in the $(k + 1)$th row is $\mu_4^k(L_{4, 3})|_{\X \to 1}$.
	}}
	\label[figure]{fig:a3p3}
\end{figure}

\end{document}